\documentclass{article}
\usepackage[utf8]{inputenc}
\usepackage{mathtools}
\usepackage{hyperref}

\usepackage{amsmath}
\usepackage{amsfonts}
\usepackage{amsthm}
\usepackage{amssymb}
\usepackage{graphicx}
\usepackage{verbatim}
\usepackage{color}
\usepackage{graphicx}
\usepackage{biblatex}

\usepackage{pgfplots}
\pgfplotsset{compat=1.14}

\theoremstyle{plain}
\newtheorem{theorem}{Theorem}

\newtheorem{corollary}[theorem]{Corollary}
\newtheorem{lemma}[theorem]{Lemma}

\newtheorem{conjecture}{Conjecture}

\theoremstyle{definition}
\newtheorem{definition}{Definition}

\newtheorem{remark}{Remark}

\DeclareMathOperator{\dom}{\text{dom}}

\title{Special Configurations in Anchored Rectangle Packings}
\author{Vincent Bian \\ \texttt{vincentbian@yahoo.com}}
\date{}

\begin{document}

\maketitle

\begin{abstract}
    Given a finite set S in $[0,1]^2$ including the origin, an anchored rectangle packing is a set of non-overlapping rectangles in the unit square where each rectangle has a point of S as its left-bottom corner and contains no point of S in its interior. Allen Freedman conjectured in the 1960's one can always find an anchored rectangle packing with total area at least $1/2$. We verify the conjecture for point configurations whose relative positions belong to certain classes of permutations.
\end{abstract}

\section{Introduction}

In general, packing problems involve fitting as many small shapes as possible into a larger one. These are useful in efficiently storing containers and in cutting shapes out of sheets of raw materials while minimizing waste. Many of these problems involve placing rectangles in a unit square, such that none of them overlap.

We consider a special packing problem where the lower left corner of all the rectangles are given. Such rectangles are called anchored. The problem was proposed by Allen Freedman, see \cite{WT}. Then it was reintroduced in many places \cite{PT}. Peter Winkler brought the attention of general public to this problem through \cite{PW_book} and \cite{PW_ACM}.

In the 1960's, it was conjectured by Allen Freedman that any set of points has an anchored rectangle packing with area at least $\frac{1}{2}$. In 2011, Dumitrescu and Tóth showed in \cite{AD} that every set of points has a packing of area at least 0.09, which was the first constant bound found, and is the best bound currently known.

We are looking at the point configuration as a permutation, depending on the value of the $y$-coordinate. We prove the conjecture for several special permutations which are illustrated in Figure~\ref{fig:six_types}. We describe the cases in more detail in the next paragraphs.

We prove the conjecture for several special cases: when the $y$-coordinates of the points form sequences with certain properties when the points are sorted by $x$-coordinate. 

Some of these cases are shown in Figure~\ref{fig:six_types}.

\begin{figure}
    \centering
    \includegraphics[scale = 0.7]{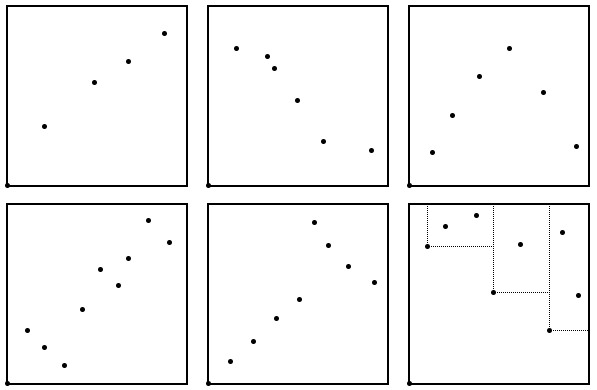}
    \caption{Increasing, decreasing, mountain, split layer, cliff, and sparse decreasing permutations, from left to right, top to bottom}
    \label{fig:six_types}
\end{figure}

We start in Section~\ref{sec:preliminaries}, providing all the definitions and basic statements. We view the point configuration excluding the origin as a permutation. We let $n$ be the number of points including origin. We have an extensive road map in Section~\ref{subsection:roadmap} after we give all the definitions.

In Section~\ref{sec:final_decreasing_run}, we study the final decreasing run, consisting of the maximal consecutive decreasing run that includes the rightmost point. We show that any packing of maximum area completely fills the area above and to the right of the points in the final decreasing run.

Next, in Section~\ref{sec:max_at_origin}, we consider the set of possible maximal rectangles anchored at the origin. We show that if the maximum area of a packing is at a local minimum with a set of points, then the areas of the possible maximal rectangles anchored at the origin are equal.

We start in Section~\ref{sec:presorted} with a special case where the beginning of a permuation is sorted. That is, the permutation starts with 1, 2, 3, $\ldots$, $m$. We discuss how to maximize the packing area when we know the best packing for the leftover permutation after. As a corollary, we get the bound of $\frac{1}{2} + \frac{1}{2n}$ for the identity permutation.

Then in Section~\ref{sec:decreasing} we cover the decreasing case and get a lower bound of $1 - \left(1 - \frac{1}{n}\right)^n$ on the area in this case.

We then discuss more general cases. First in Section~\ref{sec:split_layer} we concentrate on the case of layered permutations, that is permutations that consist of increasing sequences of decreasing runs. We show that every point set corresponding to a split layer permutation has a packing with area at least $\frac{1}{2}$.

We start with a special case of a permutation starting with a decreasing run layer of size $m$, followed by a layer of size 1.

The we prove the split-layer case, which is the layer case where between two consecutive decreasing runs at least one has the length 1.

We follow with a discussion of a cliff permutation, a special case of a split-layer permutation where only the last decreasing run can have length more than 1.

Another permutation we consider in Section~\ref{sec:mountain} is a mountain, which corresponds to a permutation that is increasing, then decreasing. We show that point sets corresponding to these mountain permutations have packings with area at least $\frac{1}{2}$.

Next, we find the minimum area for all $3! = 6$ possible permutation when $n = 4$ in Section~\ref{sec:four_dots}, explicitly showing that the minimum area is $\frac{5}{8}$, and finding the point set for which this bound is tight.

The greedy decreasing subsequence of a permutation is the subsequence of elements that are less than all preceding elements in the permutation. In Section~\ref{sec:sparse_decreasing} we consider the case where elements of this greedy decreasing subsequence are close together in the permutation, that is they are separated by fewer than $m$ points. We find lower bounds on the optimal area in this case, depending on the value of $m$.

The sparse case allows us to prove that all point sets with 9 or fewer points, including the origin, have an anchored rectangle packing with area at least $\frac{1}{2}$.

\section{Preliminaries}
\label{sec:preliminaries}

\subsection{Basic definitions and the goal}

Here is a setup for our problem \cite{WT}. 

We consider $n$ distinct points $P_i$, where $0 \leq i \leq n-1$, placed in the unit square $U = [0,1]^2$ in the plane, where one of the points is the origin $(0, 0)$.

For every point $P_k$ we choose a rectangle with sides that are parallel to axis such that the lower left corner of the rectangle is $P_k$. Such a rectangle is said to be \textit{anchored} at $P_k$. It is required that the interior of the rectangle does not contain any points or intersect with any other rectangles.

Among all possible anchored rectangle packings, we consider the one with the most area. For such a configuration, we denote a rectangle anchored at $P_k$ as $r_k$. We aim to find a lower bound on this maximal area.

Let $A(r)$ denote the area of a rectangle $r$. To repeat, our goal is to maximize the sum of the areas of all anchored rectangles:
\[ \sum_{k=0}^{n - 1} A(r_k).\]

We denote the minimum area among all sets of $n$ points as $m(n)$. The long-standing conjecture claims that $m(n)$ is at least a half \cite{WT}.

\begin{conjecture}[Main Conjecture]
\label{conjecture:main}
$m(n) \geq \frac{1}{2}$.
\end{conjecture}

It is known that this bound is tight, and the current best known lower bound is 0.09 \cite{AD}.

We can show that $m(n) \leq m(n+1)$. For any configuration $S$ with $n$ points, choose any point $P = (x, y)$. Then, add a point $P' = (x + \epsilon, y + \epsilon)$, for some small $\epsilon > 0$, to make a configuration $S'$ with $n + 1$ points. For any packing of $S'$, we can construct a packing of $S$ as follows: In the packing of $S'$, let $Q$ be the upper right corner of the rectangle anchored at $P'$. Then, delete the point $P'$, and replace the rectangle anchored at $P$ with a rectangle with upper right corner $Q$. This creates a packing of $S$.

For sufficiently small $\epsilon$, this new rectangle will not have any points on its interior. The rectangle anchored at $P$ in the packing of $S$ is strictly larger than the one anchored at $P'$ in the packing of $S'$, so the only extra area the packing of $S'$ has is the rectangle anchored at $P$, which has area at most $\epsilon$. Thus, for any packing of $S'$, there is a packing of $S$ whose area is smaller by at most $\epsilon$. As $\epsilon$ approaches 0, we see that $m(n) \leq m(n + 1)$.

Depending on the number of points, the calculations show that we can make a more precise conjecture.

\begin{conjecture}[Precise Conjecture]
\label{conjecture:precise} 
$m(n) \geq \frac{1}{2} + \frac{1}{2n}$, and this area is the maximum when the points are equally spaced on the main diagonal.
\end{conjecture}

\subsection{Maximizing and minimizing area}

As we are only interested in the maximum possible area, we only consider configurations whose areas can not be increased with small changes. Define a \textit{maximal rectangle} to be one whose upper edge touches either the upper boundary of the unit square, or another point, and whose right edge touches either the right boundary of the unit square, or another point. From now on we only consider maximal anchored rectangles. As before, we denote a maximal rectangle anchored at point $P_k$ as $r_k$.

Given a set of points, there might be several ways to draw anchored rectangles. For example, Figure~\ref{fig:two_maximal} has two different ways to use maximal rectangles and the first packing provides better area.

\begin{figure}
    \centering
    \includegraphics[scale = 0.7]{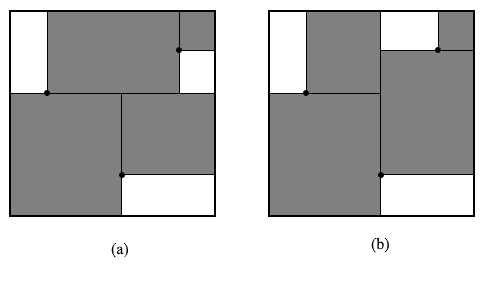}
    \caption{(a) fills 79\% of the square, while (b) fills 77\% of the square}
    \label{fig:two_maximal}
\end{figure}

A set of points is a \textit{locally minimal point set} if the maximum area can not be decreased by moving each point by a small amount. 
For example, if we have $S = \{(0, 0), (0.3, 0.4)\}$, as shown in Figure~\ref{fig:local_minimum_example}, then the maximum area is $0.82$, but when we nudge the second point right to $(0.4, 0.4)$, the maximum area drops to $0.76$.

\begin{figure}
    \centering
    \includegraphics[scale = 0.7]{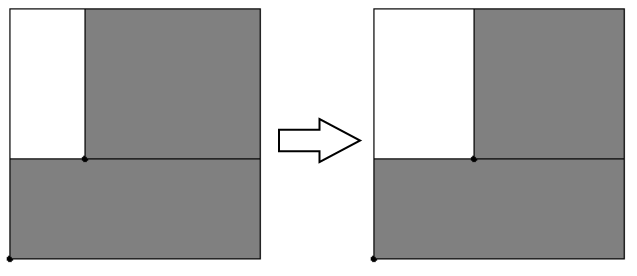}
    \caption{$S = \{(0, 0), (0.3, 0.4)\}$ is not locally minimal}
    \label{fig:local_minimum_example}
\end{figure}

As we later use the idea of packing a subrectangle, we explain the idea of scaling a packing.

\begin{lemma}
\label{lemma:scaling}
Any results about the proportion of the unit square we can fill apply to any rectangle.
\end{lemma}
\begin{proof}
Consider any anchored rectangle packing inside an $a \times b$ rectangle. Then, we can scale it vertically with scale factor $\frac{1}{a}$, and scale it horizontally with scale factor $\frac{1}{b}$. Note that this scaling preserves rectangles whose sides are parallel to the axes, and it also preserves the property that no two rectangle overlap. Also, as all areas are scaled by the same amount, the proportion of the rectangle filled is preserved as well.
\end{proof}

\subsection{Posets}

Define a poset on the points, where a point $(x_1, y_1) \preceq (x_2, y_2)$ iff $x_1 \leq x_2$ and $y_1 \leq y_2$. We say that point $(x_2, y_2)$ \emph{dominates} point $(x_1, y_1)$.

Note that if $P \preceq Q$, then the rectangles anchored at $P$ and $Q$ are independent, meaning that the rectangle anchored at $P$ can not overlap with the rectangle anchored at $Q$ without containing $Q$ in its interior. In particular, it means that any rectangle anchored at $P$ does not have an influence on the choice of a rectangle at $Q$, and vice versa.

Because of that it is useful to consider the relative order of the $y$-coordinates.
Here are some definitions with respect to ordering.

\subsection{Permutations}

We number the points left to right bottom to top and denote the points $P_0$, $P_1$, $\ldots$, $P_{n - 1}$. In particular, $P_0 = (0, 0)$.

We say a set of points $S = \{(0, 0), (x_1, y_1), (x_2, y_2), \ldots , (x_{n - 1}, y_{n - 1})\}$ satisfies a permutation $\pi$ of $\{1, 2, \ldots , n - 1\}$ if we have $P_i \preceq P_j$ iff $i \leq j$ and $\pi(i) \leq \pi(j)$. Note that this means $x_1 \leq x_2 \leq \cdots \leq x_{n - 1}$ and $y_{\pi(1)} \leq y_{\pi(2)} \leq \cdots \leq y_{\pi(n - 1)}$.

Let $M(\pi)$ denote the minimum possible maximal area when the points satisfy the permutation $\pi$.

\begin{theorem}
For any permutation $\pi$, we have $M(\pi) = M(\pi^{-1})$.
\end{theorem}

\begin{proof}
Note that the reflection of a configuration of points satisfying $\pi$ about the line $y = x$ will satisfy $\pi^{-1}$. Thus, we have a natural bijection between configurations of points satisfying $\pi$ and $\pi^{-1}$. Since reflection preserves area, the minimal possible maximum area across configurations satisfying both permutations will be the same.
\end{proof}

This fact leads to the following conjecture:

\begin{conjecture}
\label{conjecture:self_inverse}
If a permutation $\pi$ is its own inverse, then any set of points that minimizes the maximum area is symmetric about $y = x$.
\end{conjecture}

\begin{theorem}
The maximum area is a continuous function of the points.
\end{theorem}

\begin{proof}
We show that replacing a point $P(x, y)$ with $P'(x + \epsilon, y)$, for sufficiently small $|\epsilon| > 0$, changes the maximum possible area by at most $|\epsilon|$. Let the original point set with point $P$ be set $S$, and let the new set with point $P'$ be set $S'$. For every packing of $S$ with area $a$, we construct a packing of $S'$ with area between $a - |\epsilon|$ and $a + |\epsilon|$, and for every packing of $S'$ with area $a'$, we construct a packing of $S$ with area between $a' - |\epsilon|$ and $a' + |\epsilon|$.

We consider two cases, based on whether or not $P$ shares an $x$-coordinate with any other point in $S$.

\paragraph*{Case 1: $P$ does not share its $x$-coordinate \\}
Let $a$ be the area of a packing of $S$, and let $Q$ be the upper right corner of the maximal rectangle anchored at $P$.

Let $d$ be the smallest absolute difference between the $x$-coordinate of $P$ and the $x$-coordinate of another point in $S$. As long as $|\epsilon| < d$, we can replace $P$ with $P'$, and the rectangle anchored at $P$ with a rectangle anchored at $Q$ with the same upper right corner. As $|\epsilon| < d$, this new rectangle anchored at $P'$ will not have any other points in $S'$ on its interior.

In the case that $\epsilon$ is negative, and $P'$ lies inside a rectangle, we can move the right boundary of this rectangle to the left, as shown in Figure~\ref{fig:moving_point_dif_x}.

\begin{figure}
    \centering
    \includegraphics[scale = 0.7]{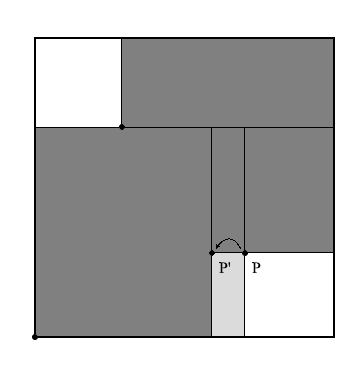}
    \caption{Moving $P$ to $P'$, and moving the right boundary of a rectangle leftwards}
    \label{fig:moving_point_dif_x}
\end{figure}

If $\epsilon$ is positive, then the only change in area is due to the decrease in the width of rectangle anchored at $P$. This changes the total area by at most $\epsilon$.

If $\epsilon$ is negative, then the rectangle anchored at $P$ gets wider, and its area increases by at most $|\epsilon|$. However, if $P'$ lies within one of the maximal rectangles in the packing of $S$, then we reduce the width of this rectangle in the packing of $S'$. As this change in width is at most $|\epsilon|$, this reduces the total area by at most $|\epsilon|$. Thus, the net change is always at most $|\epsilon|$.

Therefore, this process produces a packing of $S'$ with area between $a - |\epsilon|$ and $a + |\epsilon|$.

A similar process takes a packing of $S'$ of area $a'$ and produces a packing of $S$ with area between $a' - |\epsilon|$ and $a' + |\epsilon|$.

\paragraph*{Case 2: $P$ shares its $x$-coordinate \\}
Consider a packing of $S$. We say that the point $P$ is \textit{stuck} if $\epsilon > 0$ and $P$ is bounded on the right by a rectangle anchored at a point directly below $P$, or $\epsilon < 0$ and a point directly above $P$ is bounded on the right by a rectangle anchored at $P$. The two ways in which point $P$ can be stuck are illustrated in Figure~\ref{fig:stuck_points}. Note that unless $P$ is stuck, the procedure described in Case 1 produces a packing of $S'$ with area between $a - |\epsilon|$ and $a + |\epsilon|$.

\begin{figure}
    \centering
    \includegraphics[scale = 0.5]{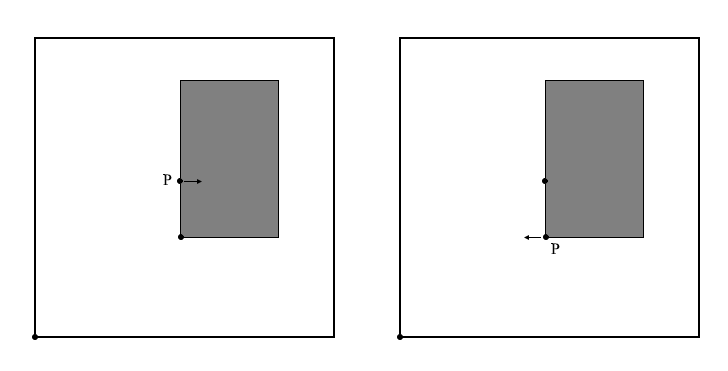}
    \caption{$\epsilon > 0$ on the left, $\epsilon < 0$ on the right}
    \label{fig:stuck_points}
\end{figure}

Now, if $P$ is stuck, we first construct another packing of $S$ with the same area, in which $P$ is not stuck. Regardless of how $P$ is stuck, there will be a rectangle directly to the right of $P$, and a point directly above or below $P$. We can split this rectangle into a part anchored at $P$, and a part anchored at the other point, as shown in Figure~\ref{fig:split_rectangle}. After splitting the rectangle, we can again use the procedure in Case 1 to produce a packing of $S'$ whose area differs by at most $|\epsilon|$.

\begin{figure}
    \centering
    \includegraphics[scale = 0.5]{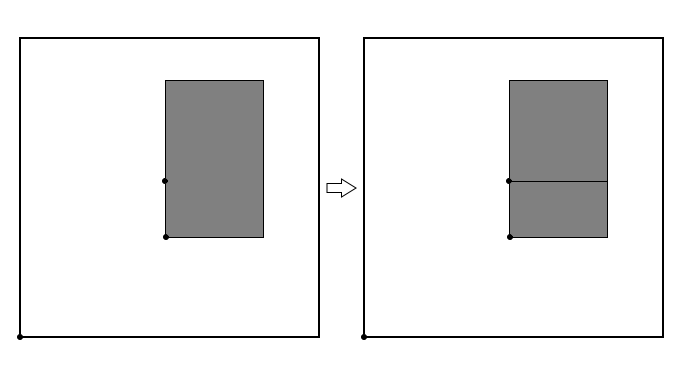}
    \caption{Depending on how $P$ is stuck, $P$ might be the upper or the lower point}
    \label{fig:split_rectangle}
\end{figure}

Finally, notice that given a packing of $S'$, the procedure detailed in Case 1 will once again produce a packing of $S$ whose area differs by at most $|\epsilon|$.

By symmetry, moving a point by $\epsilon$ in the $y$ direction, for sufficiently small $\epsilon$, will also change the maximum area by at most $|\epsilon|$.

Thus, the maximum area varies continuously as we move a single point. As any movement of the points can be written as a series of movements of each individual point, the maximum area is a continuous function of the points.
\end{proof}

We study some particular configurations of permutations. We start with some definitions.

We call a point $P_i$ in $S$ a \textit{splitting point} if $y_j < y_i$ for $j < i$, and $y_j > y_i$ for $j > i$. In terms of permutations, $\pi(j) < \pi(i)$ iff $j < i$. Note that this implies $\pi(i) = i$. If we have a splitting point, then the rectangles anchored at points $P_i$ to $P_{n - 1}$ do not interfere with the rectangles anchored at $P_1$ to $P_{i - 1}$. We can look at these two groups of points independently.

A \textit{layered permutation} is a permutation that consists of sequences of decreasing runs, so that the smallest value of the next run is greater than the largest value of the previous run. Each run is called a \textit{layer}, and the permutation is uniquely defined by the set of lengths of the layers.

We say a layered permutation is a \textit{split-layer permutation} if among any pair of consecutive layers, one has size 1. Note that a point in a layer of size 1 is a fixed point of the permutation.

For a set $S$ of points, define its \textit{dominant} point to be the minimal point $Q$ in the square such that $P \preceq Q$ for $P \in S$. Call this point  $\dom(S)$. If $S$ is a set of points corresponding to a layer, then $S = \{(x_i,y_i), \ldots, (x_j,y_j)\}$, and $\dom(S) = (x_j, y_i)$. If point $Q=(x_i,y_i)$ is a decreasing run of size 1 as a part of a layered permutation, then $Q=\dom(\{P_0, P_1, \ldots, P_i\})$.

Suppose we have a point $(x,y)$, we call a rectangle with lower left corner $(x,y)$ and upper right corner $(1,1)$ \textit{complete}.

Define the \textit{induced} packing on a rectangle $R$ with lower left corner $(x, y)$ to be an anchored packing of maximal area using $(x, y)$, and $\{P_0, P_1, \ldots, P_{n - 1}\} \cap R$.

\subsection{Road map}
\label{subsection:roadmap}

Here is an extensive road map for the rest of the paper. We study different cases depending on the permutation.

We study maximal rectangles at the origin in Section~\ref{sec:max_at_origin} and show that the points that are blocking them should be spaced in such a way that the possible maximal rectangles anchored at the origin are all the same area. This case helps with the decreasing case and the sparse decreasing case.

We consider a more general case for which the all increasing case is a special case. 

\begin{itemize}
    \item \textit{Presorted} permutations are permutations that begin with $(1, 2, \ldots, m)$ for some $1 \leq m \leq n - 1$, and are studied in Section~\ref{sec:presorted}.
    \item \textit{All increasing} permutations correspond to the identity permutation. We are interested in all increasing permutations because by the Precise Conjecture~\ref{conjecture:precise} they include the configurations of points with the smallest maximum area. We prove the Precise Conjecture~\ref{conjecture:precise} for the increasing permutation in Section~\ref{subsection:increasing}.
\end{itemize}

We then study the decreasing permutation.

\begin{itemize}
    \item The \textit{all decreasing} case corresponds to the permutation $(n - 1, n - 2, \ldots , 1)$. It is a layer permutation with exactly one layer. We show the Precise Conjecture~\ref{conjecture:precise} for this case by proving the minimum area is $1 - \left(1 - \frac{1}{n}\right)^n$ in Section~\ref{sec:decreasing}.
\end{itemize}

We discuss other more general cases. First we concentrate on the case of layered permutations. We begin with a special case of a permutation starting with a layer of size $m$, followed by a layer of size 1, possibly with other points after that.

\begin{itemize}
    \item A \textit{prelayered} permutation is one that starts with $(m, m - 1, \ldots, 1, m + 1)$, for some $m$. These are covered in Section~\ref{subsection:prelayered}. When constructing an anchored rectangle packing of maximum area, it suffices to separately maximize the area of the rectangles anchored on $P_1$, $\ldots$, $P_m$, and $P_{m + 1}$, $\ldots$, $P_{n - 1}$.
\end{itemize}

Then we prove a split-layer case.

\begin{itemize}
    \item A \textit{layer} permutation is of the form $(a_1, a_1 - 1, \ldots, 1, a_2, a_2 - 1, \ldots, a_1 + 1, \ldots, a_{\ell}, a_{\ell} - 1, \ldots, a_{\ell - 1} + 1)$, where $a_1 < a_2 < \cdots < a_{\ell}$. We say that $\ell$ is the number of layers, and the size of each layer is the number of points in it.
    \item A \textit{split-layer} permutation is a layer permutation where every pair of consecutive layers has at least one layer of size 1. We prove the Main Conjecture~\ref{conjecture:main} for point sets in this split layer case in Section~\ref{sec:split_layer}.
    \item \textit{A cliff} corresponds to the permutation $(1, 2, 3, \ldots, k, n - 1 , n - 2, \ldots, k + 1)$. This is a subcase of a split-layer permutation for which we have explicit formulae for the minimum area, and we find this formula in Section~\ref{subsection:cliff}.
\end{itemize}

All these patterns correspond to permutations that are their own inverses, except for the presorted and prelayer cases.

Another permutation we consider is a \textit{mountain}, which corresponds to a permutation that is increasing, then decreasing. Note that mountain permutations are not necessarily their own inverses, unlike many of the previously discussed cases.

\begin{itemize}
    \item A \textit{mountain} is a permutation $(a_1, \ldots, a_{n - 1})$ for which there exists some $m$ such that $a_1 < a_2 < \cdots < a_m > a_{m + 1} > \cdots > a_{n - 1}$. We show that the Main Conjecture~\ref{conjecture:main} for the mountain case in Section~\ref{sec:mountain}.
\end{itemize}

Permutations that are both split-layer and mountain permutations must be cliff permutations.

The \textit{greedy decreasing subsequence} of a permutation is the subsequence of elements that are less than all preceding elements in the permutation. Note that the elements of the greedy decreasing subsequence correspond to points which bound a possible maximal rectangle anchored at the origin. We also consider the case where elements of this greedy decreasing subsequence are close together in the permutation.

\begin{itemize}
    \item An \textit{$m$-sparse decreasing} permutation is one where there are fewer than $m$ elements between consecutive members of the greedy decreasing subsequence. We find lower bounds for the area in this case in terms of $m$ in Section~\ref{sec:sparse_decreasing}. For $m \leq 3$, we show that the Main Conjecture~\ref{conjecture:main} holds for $m$-sparse decreasing case.
\end{itemize}

The sparse case allows to prove the Main Conjecture~\ref{conjecture:main} for $n \leq 9$. We do this in Section~\ref{subsection:9_dots}.

\subsection{Examples}

We consider examples with small number of dots.

If there is only one dot, we can fill the entire square by considering it to be a rectangle anchored at the origin.

Suppose we have two dots: the origin and  $(x,y)$. The rectangle at $(x, y)$ has size $(1 - x)(1 - y)$. The other is $x \times 1$ or $1 \times y$. Thus, the maximum area is $(1 - x)(1 - y) + \max\{x, y\}$.

Without loss of generality, $x \geq y$. Then, the area is $1 - y + xy \geq 1 - y + y^2 \geq \frac{3}{4}$. Equality is achieved if and only if $x = y = \frac{1}{2}$.

\section{Final Decreasing Run}
\label{sec:final_decreasing_run}

We call \textit{the final decreasing run} to be the longest consecutive decreasing run that includes $P_{n - 1}$, that is, the sequence of points $P_j$, $P_{j + 1}$, $\ldots$, $P_{n - 1}$, such that their $y$-coordinates form a decreasing sequence, where $j$ is as small as possible.

Define the \textit{staircase region} to be the set of points that are above and to the right of at least one point in the final decreasing run.

\begin{lemma}
\label{lemma:final_staircase}
If the total area covered by the rectangles is maximized, then the entire staircase region is filled.
\end{lemma}

\begin{proof}
First, note that by splitting the staircase region into vertical strips, we can fill the entire staircase region using rectangles anchored at the points in the final decreasing run, as shown in Figure~\ref{fig:staircase_filled}.

\begin{figure}
    \centering
    \includegraphics[scale = 0.7]{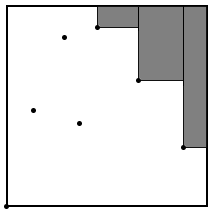}
    \caption{One of many ways to fill the staircase region}
    \label{fig:staircase_filled}
\end{figure}

Now, for any packing where rectangles anchored at points not in the final decreasing run intersect with the staircase region, we can ``cut off" these rectangles where they intersect the staircase region, as shown in Figure~\ref{fig:staircase_cutoff}.

\begin{figure}
    \centering
    \includegraphics[scale = 0.7]{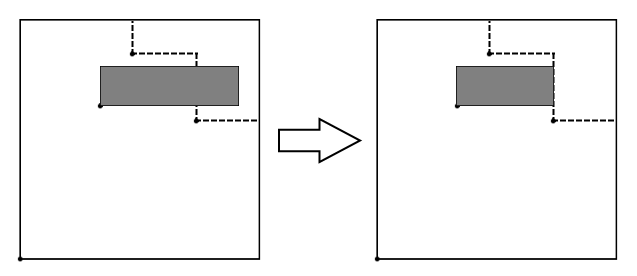}
    \caption{Keeping the part of the rectangle outside the staircase region}
    \label{fig:staircase_cutoff}
\end{figure}

Then, we can fill the staircase region using rectangles anchored at points in the final decreasing run as before. The area outside the staircase region that is covered by rectangles does not change, and the area inside the staircase region covered by rectangles can only increase. Thus, this operation can never decrease the total area, and only increases the total area if the staircase region was not already completely filled.

Therefore, in any packing with maximum area, the entire staircase region is filled.
\end{proof}

\section{Maximal rectangles at the origin}
\label{sec:max_at_origin}

We prove the following statement about maximal rectangles anchored at the origin when the points form a locally minimal point set.

Note that the set of points bounding these maximal rectangles correspond to the greedy decreasing subsequence of the permutation.

\begin{lemma}
\label{lemma:maximal_at_origin}
If the set of points is locally minimal, then the areas of the possible maximal rectangles anchored at the origin are all equal.
\end{lemma}

\begin{proof}
Note that the rectangle anchored at the origin is independent of the rectangles anchored at any other point, so any packing with maximal area must include the largest rectangle anchored at the origin. This largest rectangle must be a maximal rectangle.

Let there be $\ell$ distinct maximal rectangles anchored at the origin, and let these rectangles be $M_1$, $M_2$, $\ldots$, $M_\ell$, in order of increasing width. Note that rectangle $M_i$ is bounded on the right and $M_{i + 1}$ is bounded above by the same point. Let this point be $Q_i$.

Suppose that $A(M_i)$, the area of $M_i$, is greater than $A(M_{i + 1})$.  If we move $Q_i$ upwards by a sufficiently small amount, such that $A(M_{i + 1})$ is still less than $A(M_i)$, which remains unchanged, then we will not affect the value of $\max\{A(M_1), \ldots , A(M_\ell)\}$. However, it will decrease the area of any rectangle anchored at $Q_i$, and thus the area of the maximal packing, as shown in Figure~\ref{fig:moving_points}. This contradicts the set of points being a local minimum.

\begin{figure}[htb!]
    \centering
    \includegraphics[scale=0.7]{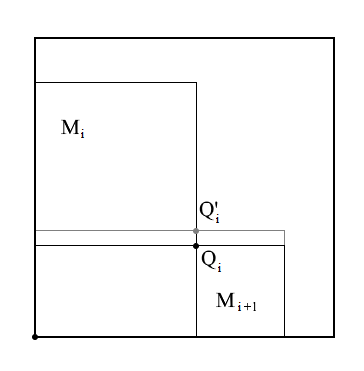}
    \caption{Moving $Q_i$ upwards}
    \label{fig:moving_points}
\end{figure}

Similarly, if $A(M_i)$ is less than $A(M_{i + 1})$, then we can move $Q_i$ rightwards by a small amount, so that $A(M_i)$ is still less than $A(M_{i + 1})$. As before, this does not change the value of $\max\{A(M_1), A(M_2), \ldots , A(M_\ell)\}$ while reducing the area of the rectangle anchored at $Q_i$, again contradicting the set of points being a local minimum.

Thus, we must have $A(M_i) = A(M_{i + 1})$, so $A(M_1) = A(M_2) = \cdots = A(M_\ell)$.
\end{proof}

\section{Presorted case}
\label{sec:presorted}

Consider a presorted permutation $(1, 2, \ldots, m, ...)$, where we can have any arrangement of the elements after $m$.

The all increasing case is a special case of the presorted case, when $m = n - 1$. Also, cliff permutations are also a special case, when the permutation is $(1, 2, \ldots, m, n - 1, n - 2, \ldots, m + 1)$.

In the presorted case, all permutations start with 1. We start by finding a lower bound on the maximum area for any permutation that starts with 1.

\begin{lemma}
\label{lemma:start1}
If the points follow a permutation that starts with $1$, and we can fill at least a proportion $k$ of the area above and to the right of $P_1$, then we can fill at least $\frac{4k - 1}{4k}$ of the square if $k \geq \frac{1}{2}$ and we can fill at least $k$ of the square if $k \leq \frac{1}{2}$. If $k \geq \frac{1}{2}$, then this bound is tight when $P_1 = \left(\frac{2k - 1}{2k}, \frac{2k - 1}{2k}\right)$, and if $k \leq \frac{1}{2}$, then this bound is tight when $P_1 = (0, 0)$.
\end{lemma}

\begin{proof}
Assume the maximum area we can fill is minimized. Then, the area of both maximal rectangles anchored at the origin must be equal. Thus, $P_1$ is on the line $y = x$. If $P_1 = (x_1, x_1)$, then the maximal area we can fill is $x_1 + (1 - x_1)^2k$. We see that the derivative of this expression is 0 at $x_1 = \frac{2k - 1}{2k}$. If $k \geq \frac{1}{2}$, then we plug in $x_1 = \frac{2k - 1}{2k}$ to get a minimal area of $1 - \frac{1}{4k}$. Otherwise, the area is minimized at $x_1 = 0$, which gives a minimum area of $k$.
\end{proof}

It follows that if $k \geq \frac{1}{2}$, then we can fill at least $\frac{1}{2}$ of the area. Now, we use induction to extend this result to permutations that start with $(1, 2, \ldots, m)$.

\begin{theorem}
\label{theorem:increasing_start}
If the points follow a permutation that starts with $(1, 2, \ldots , m)$, and we can fill at least a proportion $k$ of the area above and to the right of $P_m$, then we can fill at least \[\frac{(2m + 2)k - m}{4mk - (2m - 2)}\] of the whole square when $k \geq \frac{1}{2}$, and $k$ of the whole square when $k \leq \frac{1}{2}$. These bounds are tight when \[P_i = \left(\frac{2k - 1}{2mk - m + 1} \cdot i, \frac{2k - 1}{2mk - m + 1} \cdot i\right)\] for $1 \leq i \leq m$ when $k \geq \frac{1}{2}$, and $P_i = (0, 0)$ for $1 \leq i \leq m$ when $k \leq \frac{1}{2}$.
\end{theorem}

\begin{proof}
We prove this by induction on $m$. Note that the base case of $m = 1$ is equivalent to Lemma~\ref{lemma:start1}.

Now, for $m > 1$, note that if we restrict ourselves to the area above and to the right of $P_1$, then we get a configuration equivalent to the case of a permutation starting with $(1, 2, \ldots, m - 1)$.

We now have 2 cases, depending on whether or not $k$ is at least $\frac{1}{2}$.

\paragraph{Case 1: $k \geq \frac{1}{2}$ \\}

By the inductive hypothesis, we can fill at least \[\frac{2mk - (m - 1)}{(4m - 4)k - (2m - 4)}\] of the area above and to the right of $P_1$. If $k$ is at least $\frac{1}{2}$, then this fraction is also $\frac{1}{2}$.Then, by Lemma~\ref{lemma:start1}, we can fill at least \[1 - \frac{1}{4 \cdot \frac{2mk - (m - 1)}{(4m - 4)k - (2m - 4)}} = \frac{(2m + 2)k - m}{4mk - (2m - 2)}\] of the square, with equality if and only if $P_1 = \left(\frac{2k - 1}{2mk - m + 1}, \frac{2k - 1}{2mk - m + 1}\right)$.

Let $R$ be the rectangular region of the square that is above and to the right of $P_1$. We know from the inductive hypothesis that the bound for for $m - 1$ increasing points in $R$ is tight when $P_{i + 1}$ lies on the main diagonal of $R$, $\frac{2k - 1}{2(m - 1)k - (m - 1) + 1} \cdot i$ of the way from the bottom left to the upper right corner, for $1 \leq i \leq m - 1$.

As the lower left corner of $R$ is $P_1$, we know that a point that is a proportion $p$ of the way from the bottom left to the upper right corner of $R$ has both coordinates equal to $\frac{2k - 1}{2mk - m + 1} + p\left(1 - \frac{2k - 1}{2mk - m + 1}\right)$. Plugging in $p = \frac{2k - 1}{2(m - 1)k - (m - 1) + 1} \cdot i$, we get $P_{i + 1} = \left(\frac{2k - 1}{2mk - m + 1} \cdot (i + 1), \frac{2k - 1}{2mk - m + 1} \cdot (i + 1)\right)$ for $1 \leq i \leq m - 1$, completing the inductive step.
\paragraph{Case 2: $k \leq \frac{1}{2}$ \\}

By the inductive hypothesis, we can fill at least $k$ of the area above and to the right of $P_1$, with equality if and only if $P_2$, $P_3$, $\ldots$, $P_m$ are the same point as $P_1$. Since $k \leq \frac{1}{2}$, we have by Lemma~\ref{lemma:start1} that we can fill at least $k$ of the whole square, and this bound is tight $P_1 = (0, 0)$. However, this implies that $P_i = (0, 0)$ for all $1 \leq i \leq m$, completing the inductive step.
\end{proof}

\subsection{All increasing case}
\label{subsection:increasing}
An important special case of Theorem~\ref{theorem:increasing_start} is the all increasing case. It is conjectured that this increasing case provides the smallest area, among all possible permutations of $(1, 2, \ldots, n - 1)$.

\begin{corollary}
\label{theorem:increasing_case}
If the dots are in increasing order, then we can always fill at least $\frac{1}{2} + \frac{1}{2n}$ of the square. This bound is tight if and only if the points are $\left(\frac{i}{n}, \frac{i}{n}\right)$, for $i = 0, 1, \dots , n - 1$.
\end{corollary}

\begin{proof}
In the increasing case, we have $n - 1$ increasing points other than the origin. Also, as $P_{n - 1}$ is the uppermost and rightmost point, we can fill all the area above and to the right of it. Thus, by Theorem~\ref{theorem:increasing_start} with $m = n - 1$ and $k = 1$, we get that we can fill $\frac{n + 1}{2n} = \frac{1}{2} + \frac{1}{2n}$ of the whole square, with a tight bound if and only if $P_i = \left(\frac{i}{n}, \frac{i}{n}\right)$.
\end{proof}

Note that if we plug $k = \frac{1}{2} + \frac{1}{2n}$ into Theorem~\ref{theorem:increasing_start}, corresponding to the increasing case with $n$ points, we get $\frac{1}{2} + \frac{1}{2(m + n)}$, which is the minimum area of the increasing case with $m + n$ points, as expected.

\section{All decreasing case}
\label{sec:decreasing}

We can compute the smallest possible maximal area for a configuration with $n - 1$ decreasing points.

\begin{theorem}
\label{theorem:decreasing_case}
The minimal possible maximum area for a configuration with $n - 1$ decreasing points is $1 - \left(1 - \frac{1}{n}\right)^n$, and this bound is tight if and only if the $n - 1$ points are $\left(\left(1 - \frac{1}{n}\right)^{n - i}, \left(1 - \frac{1}{n}\right)^i\right)$ for $i = 1, 2, \ldots, n - 1$.
\end{theorem}

\begin{proof}
Consider the set of points in the decreasing case with the smallest possible maximum area. This set of points will clearly be a local minimum, so by Lemma~\ref{lemma:maximal_at_origin}, the areas of the maximal rectangles anchored at the origin must all be equal. Evaluating the areas of these rectangles, we get that
\[x_1 = x_2y_1 = x_3y_2 = \dots = x_{n - 1}y_{n - 2} = y_{n - 1}.\]

This means that $y_1 = \frac{x_1}{x_2}$, $y_2 = \frac{x_1}{x_3}$, ... , $y_{n - 2} = \frac{x_1}{x_{n - 1}}$, and $y_{n - 1} = x_1$.

The area of the maximal rectangle anchored at the origin is equal to $x_1$. By splitting the area of the staircase region into vertical strips, we get that its area is equal to $(1 - y_1)(x_2 - x_1) + (1 - y_2)(x_3 - x_2) + \dots + (1 - y_{n - 2})(x_{n - 1} - x_{n - 2}) + (1 - y_{n - 1})(1 - x_{n - 1})$.

Substituting our values of $y_k$ and expanding, we find that the area of the staircase is $1 - nx_1 + x_1\left(\frac{x_1}{x_2} +\frac{x_2}{x_3} + \dots + \frac{x_{n-2}}{x_{n - 1}} + x_{n - 1}\right)$, so the total area is 
\[1 - (n - 1)x_1 + x_1\left(\frac{x_1}{x_2} +\frac{x_2}{x_3} + \dots + \frac{x_{n-2}}{x_{n - 1}} + x_{n - 1}\right).\]

Next, if the value of $x_1$ is fixed, then we aim to minimize $\frac{x_1}{x_2} + \frac{x_2}{x_3} + \dots + \frac{x_{n-2}}{x_{n - 1}} + x_{n - 1}$. By the AM-GM inequality, we have that
\[\frac{x_1}{x_2} + \frac{x_2}{x_3} + \dots + \frac{x_{n-2}}{x_{n - 1}} + x_{n - 1} \geq (n - 1)\sqrt[n - 1]{\frac{x_1}{x_2} \cdot \frac{x_2}{x_3} \cdot \dots \cdot \frac{x_{n-2}}{x_{n - 1}} \cdot x_{n - 1}}\]
\[ = (n - 1)\sqrt[n - 1]{x_1},\]
with equality if and only if every term in the sum is equal to $x_{n - 1}$.

If these terms are all equal to $x_{n - 1}$, then we must have $x_i = x_{n - 1}^{n - i}$. Substituting this, we have that the total area is at least
\[1 - (n - 1)x_{n - 1}^{n - 1} + x_{n - 1}^{n - 1}(n - 1)x_{n - 1} = 1 + (n - 1)\left(x_{n - 1}^n - x_{n - 1}^{n - 1}\right).\]

The derivative of this expression goes from negative to positive at $x_{n - 1} = 1 - \frac{1}{n}$, so it is minimized at this value of $x_{n - 1}$.

Plugging this in to our expression for the area, we see that the minimum is $1 - \left(1 - \frac{1}{n}\right)^n$, which was to be shown.

Next, if we plug $x_i = \left(1 - \frac{1}{n}\right)^{n - i}$ into our expressions for $y_i$, we see that $y_i = \left(1 - \frac{1}{n}\right)^i$, meaning that our $n - 1$ points are indeed $\left(\left(1 - \frac{1}{n}\right)^{n - i}, \left(1 - \frac{1}{n}\right)^i\right)$ for $i = 1, 2, \ldots, n - 1$.
\end{proof}

For example, if we have $n = 3$, the area is minimized when $P_1 = \left(\frac{4}{9}, \frac{2}{3}\right)$ and $P_2 = \left(\frac{2}{3}, \frac{4}{9}\right)$. This minimum area is $\frac{19}{27}$.

Note that the points in the configuration lie on the hyperbola $xy = (1 - \frac{1}{n})^n$, as shown in Figure~\ref{fig:decreasing_hyperbola}.

\begin{figure}[htb!]
    \centering
    \includegraphics{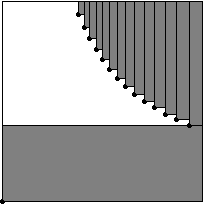}
    \caption{The configuration for $n = 15$}
    \label{fig:decreasing_hyperbola}
\end{figure}

\begin{corollary}
The minimum area in the decreasing case is more than the minimum area in the increasing case.
\end{corollary}

\begin{proof}
This statement is equivalent to $1 - \left(1 - \frac{1}{n}\right)^n \geq \frac{1}{2} + \frac{1}{2n}$. Note that it is true for $n = 1$, so assume $n \geq 2$.

Consider the inequality
\[\frac{1}{2} \geq \left(1 - \frac{1}{n}\right)^{n - 1}.\]

It holds true for $n = 2$, and it is well known that $\left(1 - \frac{1}{n}\right)^{n - 1}$ is a decreasing sequence, so this inequality is true for all $n \geq 2$.

Multiplying both sides by $\left(1 - \frac{1}{n}\right)$, we see that
\[\frac{1}{2}\left(1 - \frac{1}{n}\right) = \frac{1}{2} - \frac{1}{2n} \geq \left(1 - \frac{1}{n}\right)^n.\]

Rearranging, we get the desired $1 - \left(1 - \frac{1}{n}\right)^n \geq \frac{1}{2} + \frac{1}{2n}$.
\end{proof}

\subsection{Cliff case}
\label{subsection:cliff}
We can combine the result for decreasing permutations with Theorem~\ref{theorem:increasing_start} to find the minimum area in the cliff case by plugging in $m$ as the number of increasing points and $k$ as the minimum area for the decreasing points.

\begin{corollary}
\label{corollary:cliff_case}
In the cliff case, if the permutation is $(1, 2, \ldots, m, n - 1, n - 2, n - 3, \ldots, m + 1)$, the minimum area is
\[\frac{(2m + 2)k - m}{4mk - (2m - 2)},\] where $k = 1 - \left(1 - \frac{1}{n - m}\right)^{n - m}$.
\end{corollary}

\begin{proof}
The area above and to the right of $P_m$ is equivalent to the decreasing case with $n - m$ points, so by Theorem~\ref{theorem:decreasing_case}, we can fill at least $1 - \left(1 - \frac{1}{n - m}\right)^{n - m} = k$ of this area.

Then, since this configuration starts with $m$ increasing points, we have by Theorem~\ref{theorem:increasing_start} that we can fill at least $\frac{(2m + 2)k - m}{4mk - (2m - 2)}$ of the whole square.
\end{proof}

\section{Split-layer permutation case}
\label{sec:split_layer}

A \emph{layer permutation} is a series of decreasing runs, where the lowest element in each run is higher than the highest element in the previous run. Equivalently, we can divide the permutation into consecutive blocks, where each block is a decreasing run, and if we take one element from each block, we get an increasing sequence.

We define a \emph{split-layer permutation} to be a layer permutation where among two consecutive decreasing runs, one of them has length 1. For example, $(3, 2, 1, 5, 4)$ is a layer permutation, but is not a split-layer permutation, as the first two decreasing runs have lengths 3 and 2, respectively. However, $(3, 2, 1, 4, 6, 5, 7)$ is a split-layer permutation, as the runs have lengths 3, 1, 2, and 1.

We now show that the splitting points, aka the runs of length 1, in a split-layer permutation create independent groups of points, in that we can consider the two groups independently.

\begin{theorem}
\label{theorem:independent_layers}
If the permutation starts with $(m, m - 1, \ldots, 1)$, and $P_{m + 1}$ is the first splitting point, then the packing is minimal only if the induced packing with rectangles anchored at points  $\{P_{m + 1}, P_{m + 2}, \ldots, P_{n - 1}\}$ in the full rectangle anchored at $\dom(\{P_1, P_2, \ldots, P_m\})$ is minimal.
\end{theorem}

\begin{proof}
Let the rectangular region above and to the right of the dominating point $\dom\left(\{P_1, P_2, \ldots, P_m\}\right) = (x_m, y_1)$ be $R$. Then, considering $R$ independently, and adding the lower left corner, or ``origin", $(x_m, y_1)$ of $R$, we can select rectangles anchored at points $P_{m + 1}$, $P_{m + 2} \ldots, P_{n - 1}$ to minimize the proportion of $R$ that can be filled.

Also, outside of $R$, the most area we can fill is a rectangle anchored at the origin, followed by the staircase region defined by $P_1$, $P_2$, $\ldots$, $P_m$.

Then, because the maximal rectangle anchored at the ``origin" of $R$ is either a vertical rectangle extending to the top of $R$, or a horizontal one extending to the right edge of $R$, it can be combined with the rectangle anchored at $P_1$ or $P_{m - 1}$, respectively, as shown in Figure~\ref{fig:layers_independent}.

\begin{figure}
    \centering
    \includegraphics{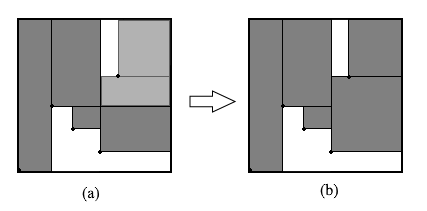}
    \caption{Example with $m = 3$. The upper right part of (a) can be incorporated into the rest of the diagram, as shown in (b)}
    \label{fig:layers_independent}
\end{figure}
\end{proof}

We now use these ideas to show that the conjecture holds for the split-layer case. Split-layer permutations either start with a decreasing run with one point, or a decreasing run of multiple points followed by a run of one point. We have already considered the case of the first run having one point in Lemma~\ref{lemma:end1}, so we consider the prelayered case, when the permutation starts with a decreasing run, followed by a point, to set up an induction argument.

\subsection{Prelayered case}
\label{subsection:prelayered}

We consider a case of two decreasing runs, where the second run consist of one point.

We first prove the following lemma:

\begin{lemma}
\label{lemma:decreasing_half_last}
If the permutation starts with $(m, m - 1, \ldots, 1)$, and we can fill at least $\frac{1}{2}$ of the rectangle $R$  with opposite vertices $P_m$ and $(1, 1)$, as shown in Figure~\ref{fig:decreasing_half_last}, and $x_1 \geq 2\left(1 - \frac{1}{2m}\right)^m - 1$, then we can fill at least $\frac{1}{2}$ of the whole square.
\end{lemma}

\begin{figure}
    \centering
    \includegraphics[scale = 0.7]{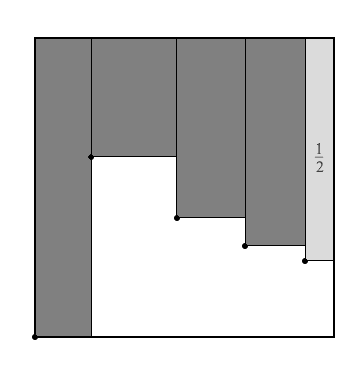}
    \caption{Case $m = 4$}
    \label{fig:decreasing_half_last}
\end{figure}

\begin{proof}
First, by Lemma~\ref{lemma:maximal_at_origin}, we have $x_1 = y_1x_2 = y_2x_3 = \cdots = y_{m - 1}x_m = y_m$, so $P_a = \left(x_a, \frac{x_1}{x_{a + 1}}\right)$ for $a < m$, and $P_m = \left(x_m, x_1\right)$.

Then, by the area computation in Theorem~\ref{theorem:decreasing_case}, the area of the whole square we could fill, if we could fill \emph{all} of rectangle $R$, is
\[1 - mx_1 + x_1\left(\frac{x_1}{x_2} + \frac{x_2}{x_3} + \cdots + \frac{x_{m - 1}}{x_m} + x_m\right).\]
Subtracting half the area of $R$, we see that the total area is
\[1 - mx_1 + x_1\left(\frac{x_1}{x_2} + \frac{x_2}{x_3} + \cdots + \frac{x_{m - 1}}{x_m} + x_m\right) - \frac{1}{2}(1 - x_m)(1 - x_1).\]

Now, note that for fixed $x_1$ and $x_m$, $\frac{x_1}{x_2} + \frac{x_2}{x_3} + \cdots + \frac{x_{m - 1}}{x_m}$ is minimized when
\[\frac{x_1}{x_2} = \frac{x_2}{x_3} = \cdots = \frac{x_{m - 1}}{x_m} = \sqrt[m - 1]{\frac{x_1}{x_m}},\]
so the area is at least
\[1 - mx_1 + x_1\left((m - 1)\sqrt[m - 1]{\frac{x_1}{x_m}} + x_m\right) - \frac{1}{2}(1 - x_m)(1 - x_1).\]

Expanding the last term and simplifying, we see that the area is at least
\[\frac{1}{2} - \left(m - \frac{1}{2}\right)x_1 + \frac{1}{2}x_m - \frac{1}{2}x_1x_m + x_1\left((m - 1)\sqrt[m - 1]{\frac{x_1}{x_m}} + x_m\right).\]

We aim to show that this area is at least $\frac{1}{2}$, so it suffices to show
\[-\left(m - \frac{1}{2}\right)x_1 + \frac{1}{2}x_m - \frac{1}{2}x_1x_m + x_1\left((m - 1)\sqrt[m - 1]{\frac{x_1}{x_m}} + x_m\right) \geq 0.\]

Rearranging and doubling both sides, we see that this is equivalent to
\[x_m + (2m - 2)x_1^{\frac{m}{m - 1}}x_m^{-\frac{1}{m - 1}} + x_1x_m \geq (2m - 1)x_1.\]

Now, for fixed $x_1$, the left side is minimized when the partial derivative with respect to $x_m$ is 0. This partial derivative is
\begin{align*}
    &1 + x_1 + (2m - 2)\left(\frac{x_1}{x_m}\right)^{\frac{m}{m - 1}} \cdot \left(-\frac{1}{m - 1}\right) \\
    &= 1 + x_1 - 2\left(\frac{x_1}{x_m}\right)^{\frac{m}{m - 1}}.
\end{align*}

Thus, we have \[1 + x_1 - 2\left(\frac{x_1}{x_m}\right)^{\frac{m}{m - 1}} = 0,\] or equivalently, \[x_m = \frac{x_1}{\left(\frac{1 + x_1}{2}\right)^{\frac{m - 1}{m}}}.\]

Plugging this into our inequality, we see that it suffices to show
\[\frac{x_1}{\left(\frac{1 + x_1}{2}\right)^{\frac{m - 1}{m}}} + (2m - 2)x_1^{\frac{m}{m - 1}}\left(\frac{x_1}{\left(\frac{1 + x_1}{2}\right)^{\frac{m - 1}{m}}}\right)^{-\frac{1}{m - 1}} + x_1 \cdot \frac{x_1}{\left(\frac{1 + x_1}{2}\right)^{\frac{m - 1}{m}}} \geq (2m - 1)x_1.\]

If $x_1 = 0$, then both sides are 0 and the inequality is true. Thus, we may assume $x_1 > 0$, so we can divide both sides by $x_1$ to obtain
\[\frac{1}{\left(\frac{1 + x_1}{2}\right)^{\frac{m - 1}{m}}} + (2m - 2)\left(\frac{1}{\left(\frac{1 + x_1}{2}\right)^{\frac{m - 1}{m}}}\right)^{-\frac{1}{m - 1}} + x_1 \cdot \frac{1}{\left(\frac{1 + x_1}{2}\right)^{\frac{m - 1}{m}}} \geq 2m - 1.\]

Simplifying, we see that the left hand side is equal to \[2m\left(\frac{1 + x_1}{2}\right)^{\frac{1}{m}}.\]

However, we assumed that $x_1 \geq 2\left(1 - \frac{1}{2m}\right)^{m} - 1$, so we have \[2m\left(\frac{1 + x_1}{2}\right)^{\frac{1}{m}} \geq 2m\left(1 - \frac{1}{2m}\right) = 2m - 1.\]

Therefore, we can fill at least $\frac{1}{2}$ of the square.
\end{proof}

Now, we use the asymptotic behavior of $\left(1 - \frac{1}{2m}\right)^m$ to show that $x_1 \geq 2e^{-\frac{1}{2}} - 1$ work for any $m$.

\begin{corollary}
\label{corollary:decreasing_half_last}
If the permutation starts with $(m, m - 1, \ldots, 1)$, and if we can fill at least $\frac{1}{2}$ of the rectangle with opposite vertices $P_m$ and $(1, 1)$, and $x_1 \geq 2e^{-\frac{1}{2}} - 1 \approx 0.2131$, then we can fill at least $\frac{1}{2}$ of the whole square.
\end{corollary}

\begin{proof}
Note that $\left(1 - \frac{1}{2m}\right)^m$ is an increasing function of $m$, and approaches $e^{-\frac{1}{2}}$.

Thus, given $x_1 \geq 2e^{-\frac{1}{2}} - 1 \geq 2\left(1 - \frac{1}{2m}\right)^m - 1$, we can fill at least half of the whole square.
\end{proof}

We now prove the Main Conjecture~\ref{conjecture:main} in the case of a prelayered permutation for which the first layer has 2 points under the assumption that we can fill at least half the area above and to the right of the dominating point $(x_2, y_1)$ of $P_1$ and $P_2$.

\begin{lemma}
\label{lemma:split_layer_2_points}
Suppose the points satisfy a split-layer permutation starting with $(2, 1)$, and let $R$ be the full rectangle anchored at $\dom\{P_1, P_2\}$, and we can fill $\frac{1}{2}$ of $R$ with an induced packing of $R$. Then, we can fill at least $\frac{1}{2}$ of the whole square.
\end{lemma}

\begin{proof}
Suppose the maximum area is minimized. Then, by Lemma~\ref{lemma:maximal_at_origin}, we have $x_1 = x_2y_1 = y_2$, as shown in Figure~\ref{fig:split_layer_2_points}, so $P_1 = \left(x_1, \frac{x_1}{x_2}\right)$, and $P_2 = (x_2, x_1)$.

\begin{figure}
    \centering
    \includegraphics[scale = 0.7]{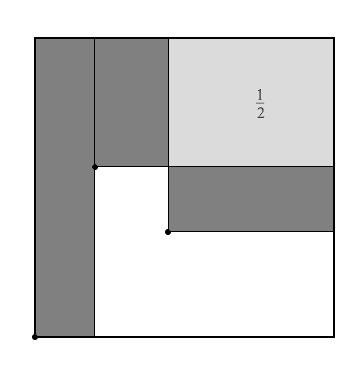}
    \caption{The first layer with 2 points}
    \label{fig:split_layer_2_points}
\end{figure}

By the area computation in Theorem~\ref{theorem:decreasing_case}, the area we can fill, if we could fill all of $R$, is 
$1 - 2x_1 + x_1\left(\frac{x_1}{x_2} + x_2\right).$

Removing half the area of $R$ from this expression, we see that the area we can fill is
\begin{align*}
    &1 - 2x_1 + x_1\left(\frac{x_1}{x_2} + x_2\right) - \frac{1}{2}(1 - x_2)\left(1 - \frac{x_1}{x_2}\right) \\
    &= \frac{1}{2} - \frac{5}{2}x_1 + \frac{1}{2}x_2 + \frac{1}{2}\frac{x_1}{x_2} + \frac{x_1^2}{x_2} + x_1x_2.
\end{align*}

By the AM-GM inequality, we have
\[\frac{\frac{1}{2}x_2 + \frac{1}{2}\frac{x_1}{x_2} + \frac{x_1^2}{x_2} + x_1x_2}{4} \geq \sqrt[4]{\frac{1}{2}x_2 \cdot \frac{1}{2}\frac{x_1}{x_2} \cdot \frac{x_1^2}{x_2} \cdot x_1x_2} = \frac{x_1}{\sqrt{2}}.\]

Thus,
\[\frac{1}{2}x_2 + \frac{1}{2}\frac{x_1}{x_2} + \frac{x_1^2}{x_2} + x_1x_2 \geq 2\sqrt{2}x_1,\]
so the area is
\[\frac{1}{2} - \frac{5}{2}x_1 + \frac{1}{2}x_2 + \frac{1}{2}\frac{x_1}{x_2} + \frac{x_1^2}{x_2} + x_1x_2 \geq \frac{1}{2} - \frac{5}{2}x_1 + 2\sqrt{2}x_1 \geq \frac{1}{2},\]
meaning that the total area is at least $\frac{1}{2}$.
\end{proof}

We now use these results to show that we can fill at least half the square in the case of a split-layer permutation when the first layer has more than 2 points.

\begin{lemma}
\label{lemma:split_layer_start}
Suppose the permutation begins with $(m, m - 1, \ldots, 1, m + 1)$, for $m > 2$, and let $R$ be the full rectangle anchored at the dominating point of $\{P_1, P_2, \ldots, P_m\}$, and we can fill $\frac{1}{2}$ of $R$ with an induced packing of $R$, then we can fill at least $\frac{1}{2}$ of the whole square.
\end{lemma}

\begin{proof}
By a similar method as in Theorem~\ref{theorem:decreasing_case}, we separately fill the rectangle $R$ and the rest of the square.

Consider the induced packing of rectangle $R$. As $P_{m + 1}$ is a splitting point, it is the only point bounding the maximal rectangles anchored at the lower left corner of $R$. Each of these rectangles can be incorporated into one of the other rectangles. Thus, if the induced packing of $R$ can cover a proportion of $R$, then a packing in the original unit square can also fill that proportion of $R$.

Now we can solve the problem while ignoring the points $P_{m + 1}$, $P_{m + 2}$, $\ldots$, $P_{n - 1}$, and then subtracting the part of the rectangle $R$ that we cannot fill: $\frac{1}{2}(1 - y_1)(1 - x_m)$.

By the area computation in Theorem~\ref{theorem:decreasing_case}, without the point $P_{m+1}$, the amount of area we can fill is
\[1 - mx_1 + x_1\left(\frac{x_1}{x_2} + \frac{x_2}{x_3} + \cdots + \frac{x_{m - 1}}{x_m} + x_m\right).\]
Therefore, the maximum area we can fill is
\[1 - mx_1 + x_1\left(\frac{x_1}{x_2} + \frac{x_2}{x_3} + \cdots + \frac{x_{m - 1}}{x_m} + x_m\right) - \frac{1}{2}\left(1 - \frac{x_1}{x_2}\right)(1 - x_m).\]

If the three values $x_1$, $x_2$, and $x_m$ are fixed, then we aim to minimize the value of $\frac{x_2}{x_3} + \frac{x_3}{x_4} + \cdots + \frac{x_{m - 1}}{x_m}$. By the AM-GM inequality, this is minimized when $\frac{x_2}{x_3} = \frac{x_3}{x_4} = \cdots = \frac{x_{m - 1}}{x_m}$. Let this common value be $r$. Then, we have 
\[x_2 = r^{m - 2}x_m.\]

Plugging this into the expression for the area and simplifying, we get that the maximum area is 
\[\frac{1}{2} - mx_1 + \frac{x_1^2}{r^{m - 2}x_m} + (m - 2)rx_1 + x_1x_m + \frac{x_1}{2r^{m - 2}x_m} + \frac{1}{2}x_m - \frac{x_1}{2r^{m - 2}}.\]

Taking the partial derivative with respect to $x_m$ and setting it to 0, we see that 
\[-\frac{x_1^2}{r^{m - 2}x_m^2} + x_1 - \frac{x_1}{2r^{m - 2}x_m^2} + \frac{1}{2} = 0.\]
We can factor this as
\[\left(x_1 + \frac{1}{2}\right)\left(1 - \frac{x_1}{r^{m - 2}x_m^2}\right) = 0.\]
Because $x_1 + \frac{1}{2}$ is always positive, we have that
\[x_1 = r^{m - 2}x_m^2.\]
Plugging this into our expression for the total area, we get that the area is
\[\frac{1}{2} - mx_m^2r^{m - 2} + 2x_m^3r^{m - 2} + (m - 2)x_m^2r^{m - 1} +  x_m - \frac{1}{2}x_m^2.\]
Taking the partial derivative with respect to $r$ and setting it to 0, we must have
\[-m(m - 2)x_m^2r^{m - 3} + 2(m - 2)x_m^3r^{m - 3} + (m - 2)(m - 1)x_m^2r^{m - 2} = 0.\]
Dividing out $(m - 2)x_m^2r^{m - 3}$, we have that $-m + 2x_m + (m - 1)r = 0$. Thus, the partial derivative is 0 when
\[r = \frac{m - 2x_m}{m - 1}.\]

Note that as $r \leq 1$, if $x_m \leq \frac{1}{2}$, we have $\frac{m - 2x_m}{m - 1} \geq 1$, so in this case, we have that $r = 1$ minimizes the area.

We now have two cases, depending on the value of $x_m$.

\paragraph*{Case 1: $x_m \leq \frac{1}{2}$\\}

In this case, we must have $r = 1$. Then, we know that $x_1 = r^{m - 2}x_m^2 = x_m^2$, and $x_2 = x_3 = \cdots = x_m$.

Then, as shown in Figure~\ref{fig:split_layer_r1}, this is equivalent to two decreasing points $P_1$ and $P_m$, where we can fill at least $\frac{1}{2}$ of the area above and to the right of the dominating points of $\{P_1, P_m\}$. This is simply the split-layer case with 2 points in the first layer, which was resolved in Lemma~\ref{lemma:split_layer_2_points}.

\begin{figure}
    \centering
    \includegraphics[scale = 0.7]{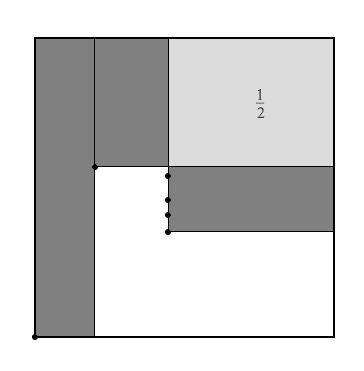}
    \caption{Example when $r = 1$ and $m = 5$}
    \label{fig:split_layer_r1}
\end{figure}

\paragraph*{Case 2: $x_m > \frac{1}{2}$\\}

Note that the dominating point of $\{P_1, P_2, \ldots, P_m\}$ is $(x_m, y_1)$, so rectangle $R$ has opposite vertices $(x_m, y_1)$ and $(1, 1)$. As $y_1 \geq y_m$, we have that the interior of $R$ is a subset of the region $Q$ above and to the right of $P_m$.

We can fill all of $Q - R$, the part of $Q$ that is not in $R$, as shown in Figure~\ref{fig:split_layer_half_last}. Also, we can fill at least half of $R$. Thus, we can fill at least half of $Q$.

\begin{figure}
    \centering
    \includegraphics[scale = 0.7]{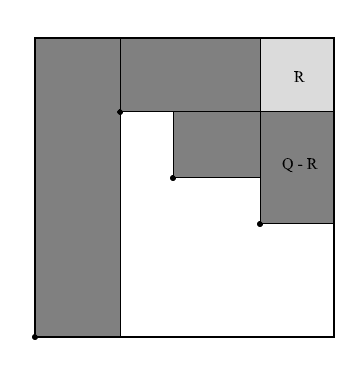}
    \caption{$R$ and $Q - R$ when $m = 3$}
    \label{fig:split_layer_half_last}
\end{figure}

Thus, by Corollary~\ref{corollary:decreasing_half_last}, it suffices to show that $x_1 \geq 2e^{-\frac{1}{2}} - 1$.

Note that
\[x_1 = r^{m - 2}x_m^2 = x_m^2\left(\frac{m - 2x_m}{m - 1}\right)^{m - 2}.\]

The derivative of the right hand side expression with respect to $x_m$ is
\begin{align*}
    &2x_m\left(\frac{m - 2x_m}{m - 1}\right)^{m - 2} + x_m^2 \cdot \frac{-2}{m - 1} \cdot (m - 2)\left(\frac{m - 2x_m}{m - 1}\right)^{m - 3} \\
    &= 2x_m\left(\frac{m - 2x_m}{m - 1}\right)^{m - 3} \cdot \frac{m}{m - 1}(1 - x_m) \\
    &\geq 0,
\end{align*}
as every term in the last product is nonnegative.

Thus, $x_1$ is a nondecreasing function of $x_m$, so the minimum value of $x_1$ occurs at the minimum value of $x_m$, which is $\frac{1}{2}$. Plugging in $x_m = \frac{1}{2}$, we see that $\frac{1}{4} \geq 2e^{-\frac{1}{2}} - 1$ is the minimum value of $x_1$.

Therefore, by Corollary~\ref{corollary:decreasing_half_last}, we can fill at least half of the whole square.
\end{proof}

We now use these results to resolve the split-layer case.

\begin{theorem}
\label{theorem:split_layer_case}
When the points follow a split-layer permutation, we can fill at least half the square.
\end{theorem}

\begin{proof}
We prove this by strong induction on $n$. First, for the base case $n = 1$, we can fill the entire square.

Now, for general $n$, let $m$ be the length of the first layer.

\paragraph*{Case 1: $m = 1$ \\}
Note that the region above and to the right of $P_1$ forms the split-layer case with $n - 1$ points. By the inductive hypothesis, we can fill at least half of this area. Then, by Lemma~\ref{lemma:start1}, we can fill at least $\frac{1}{2}$ of the whole square as well.

\paragraph*{Case 2: $m > 1$ \\}
If the permutation only has one layer, then it is a decreasing permutation. Then, by Theorem~\ref{theorem:decreasing_case}, we can fill at least half the square.

Otherwise, since it is a split-layer permutation, the next layer only has one point.

Thus, the permutation starts with $(m, m - 1, \ldots, 2, 1, m + 1)$. Now, let $P$ be the dominating point of $\{P_1, P_2, \ldots, P_m\}$, and let $R$ be the rectangle with opposite vertices $P$ and $(1, 1)$. Consider an induced packing of $R$, and note that if the packing is maximal, the rectangle anchored at the origin of $R$ either reaches the top edge or the right edge of $R$. Thus, as shown in Figure~\ref{fig:layers_independent}, this rectangle can be incorporated into the rectangle anchored at either $P_1$ if it is vertical, or the rectangle anchored at $P_m$, if it is horizontal.

Now, note that the points in $R$, along with the lower left corner of $R$, form the split-layer case with $n - m$ points. Thus, by the inductive hypothesis, we can fill at least half of $R$ with an induced packing.

Therefore, by Lemma~\ref{lemma:split_layer_start}, we can fill at least half of the whole square.
\end{proof}

\section{Mountain case}
\label{sec:mountain}

Mountain permutations are those that are increasing, and then decreasing, such as $(1, 3, 5, 7, 8, 6, 4, 2)$.

In this section, we show that the maximum area in the mountain case is always greater than $\frac{1}{2}$.

\begin{lemma}
\label{lemma:end1}
Suppose the permutation starts with $m$, and ends with $(m - 1, m - 2, \ldots, 1)$, for some $m \leq n$ and we can fill at least $\frac{1}{2}$ of the rectangle with opposite vertices $P_1$ and $(x_{n - m + 1}, 1)$, as shown in Figure~\ref{fig:mountain_end}. Then, we can fill at least $\frac{1}{2}$ of the entire square.
\end{lemma}

\begin{proof}
First, assume that the maximum area is minimized. Note that the $m + 1$ maximal rectangles anchored at the origin are bounded by one of the $m$ points in the set $D = \{P_1, P_{n - m + 1}, P_{n - m + 2}, \ldots, P_{n - 1}\}$, as shown in Figure~\ref{fig:mountain_end}. Thus, by Lemma~\ref{lemma:maximal_at_origin}, we have that the areas of these rectangles are equal, so
\[x_1 = y_1x_{n - m + 1} = y_{n - m + 1}x_{n - m + 2} = \cdots = y_{n - 2}x_{n - 1} = y_{n - 1}.\]

\begin{figure}
    \centering
    \includegraphics[scale = 0.7]{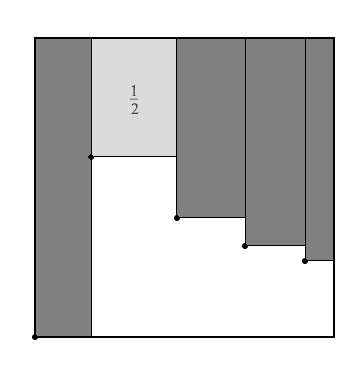}
    \caption{Example with $m = 4$}
    \label{fig:mountain_end}
\end{figure}

By a similar calculation as in the proof of Theorem~\ref{theorem:decreasing_case}, if we assume we can fill the entire staircase region defined by the points in $D$, the total area we can fill is
\[1 - mx_1 + x_1\left(\frac{x_1}{x_{n - m + 1}} + \frac{x_{n - m + 1}}{x_{n - m + 2}} + \cdots + \frac{x_{n - 2}}{x_{n - 1}} + x_{n - 1}\right).\]

However, we can only fill $\frac{1}{2}$ of the area of the rectangle with opposite vertices $P_1$ and $(x_{n - m + 1}, 1)$, so the total area we can fill is at least
\[1 - mx_1 + x_1\left(\frac{x_1}{x_{n - m + 1}} + \cdots + x_{n - 1}\right) - \frac{1}{2}\left(1 - \frac{x_1}{x_{n - m + 1}}\right)(x_{n - m + 1} - x_1).\]

Now, assuming that $x_1$ and $x_{n - m + 1}$ are fixed, we have from a similar argument as in the proof of Theorem~\ref{theorem:decreasing_case} that the area is minimized when
\[\frac{x_{n - m + 1}}{x_{n - m + 2}} = \cdots = \frac{x_{n - 2}}{x_{n - 1}} = x_{n - 1},\]
meaning that
\[x_{n - a} = x_{n - 1}^a\]
for $1 \leq a \leq m - 1$. In particular, we use the fact that $x_{n - m + 1} = x_{n - 1}^{m - 1}$.

Plugging this in and simplifying, we get that the area is
\[1 - (m - 1)x_1 + \frac{x_1^2}{2x_{n - 1}^{m - 1}} + (m - 1)x_1x_{n - 1} - \frac{1}{2}x_{n - 1}^{m - 1}.\]

The partial derivative of this expression with respect to $x_1$ is
\[-(m - 1) + \frac{x_1}{x_{n - 1}^{m - 1}} + (m - 1)x_{n - 1}.\]

This goes from negative to positive at
\[x_1 = (m - 1)x_{n - 1}^{m - 1}(1 - x_{n - 1}),\]
so the area is minimized at this value of $x_1$.

Plugging this into the expression for the area and simplifying, we get that the area is at least
\[1 - \frac{1}{2}x_{n - 1}^{m - 1}\left(1 + (m - 1)^2(1 - x_{n - 1})^2\right).\]

We now show that
\[x_{n - 1}^{m - 1}\left(1 + (m - 1)^2(1 - x_{n - 1})^2\right)\]
is a nondecreasing function of $x_{n - 1}$, by showing that its derivative is nonnegative.

This expression can be written as $(m - 1)^2x^{m - 1}(1 - x)^2 + x^{m - 1}$. Then, we see that its derivative is
\[(m - 1)^2\left((m - 1)x^{m - 2}(1 - x)^2 - 2x^{m - 1}(1 - x)\right) + (m - 1)x^{m - 2}.\]

We can factor this as
\[(m - 1)x^{m - 2}\left((m - 1)^2(1 - x)^2 - 2(m - 1)x(1 - x) + 1\right).\]

Because $(m - 1)x^{m - 2}$ is nonnegative, it suffices to show that
\[(m - 1)^2(1 - x)^2 - 2(m - 1)x(1 - x) + 1\]
is nonnegative. Since $x \leq 1$, we have
\begin{align*}
&(m - 1)^2(1 - x)^2 - 2(m - 1)x(1 - x) + 1 \\
&\geq (m - 1)^2(1 - x)^2 - 2(m - 1)(1 - x) + 1 \\
&= \left((m - 1)(1 - x) - 1\right)^2 \\
&\geq 0.
\end{align*}

Thus, we have that
\[x_{n - 1}^{m - 1}\left(1 + (m - 1)^2(1 - x_{n - 1})^2\right)\]
is nondecreasing, so we know that the area is a nonincreasing function of $x_{n - 1}$.

Since $0 \leq x_{n - 1} \leq 1$, we have that the minimum value is attained at $x_{n - 1} = 1$. Plugging in this value of $x_{n - 1}$, we get that the area is at least $\frac{1}{2}$.
\end{proof}

Using this result, we can resolve the mountain case.

\begin{theorem}
If the points follow a mountain permutation, we can always fill at least $\frac{1}{2}$ of the square.
\end{theorem}

\begin{proof}
We prove this by strong induction on $n$. For our base case of $n = 1$, we can fill the whole square.

Now, consider the mountain case with $n$ points. We have 2 cases: the permutation starts with 1, or the permutation ends with 1.

\paragraph*{Case 1: permutation starts with 1 \\}

Consider the region above and to the right of $P_1$. Note that the points within it form exactly the mountain case with $n - 1$ points, so we have by the inductive hypothesis that we can fill at least $\frac{1}{2}$ of this region.

Then, by Lemma~\ref{lemma:start1}, we can fill at least $\frac{1}{2}$ of the whole square.

\paragraph*{Case 2: permutation ends with 1 \\}

Note that if the permutation starts with $m$, then it must end with $(m - 1, m - 2, \ldots, 1)$. Now, consider the rectangular region with opposite vertices $P_1$ and $(x_{n - m + 1}, 1)$. 
The points in this region, including those on the border, form the mountain case with $n - m$ points. Thus, by the inductive hypothesis, we can fill at least $\frac{1}{2}$ of this region.

Then, by Lemma~\ref{lemma:end1}, we can fill at least $\frac{1}{2}$ of the whole square.
\end{proof}

It seems like among all permutations, increasing permutations have the smallest possible optimal area, as stated in the Precise Conjecture~\ref{conjecture:precise}.

Thus, we also propose Conjecture~\ref{conjecture:mountain_end_minimum}, which implies that the area of a mountain permutation with $n$ points is at least $\frac{1}{2} + \frac{1}{2n}$.

\begin{conjecture}
\label{conjecture:mountain_end_minimum}
Suppose the permutation starts with $m$, and ends with $(m - 1, m - 2, \ldots, 1)$, for some $m \leq n$ and we can fill a proportion $k > \frac{1}{2}$ of the rectangle with opposite vertices $P_1$ and $(x_{n - m + 1}, 1)$, as shown in Figure~\ref{fig:mountain_end}. Then, we can fill at least $\frac{1}{2} + \frac{1}{2\left(m + \frac{1}{2k - 1}\right)}$ of the entire square.
\end{conjecture}

A computation similar to the one in Lemma~\ref{lemma:end1} shows that Conjecture~\ref{conjecture:mountain_end_minimum} is equivalent to the inequality
\[k\frac{x^2}{y} + (m - 1)x\sqrt[m - 1]{y} + \left(\frac{1}{2} - \frac{1}{2\left(m + \frac{1}{2k - 1}\right)}\right) - (m - 2 + 2k)x - (1 - k)y \geq 0,\]
for $k \geq \frac{1}{2}$, $m \geq 2$, and $0 \leq x \leq y \leq 1$, which is strongly supported by numerical evidence.

\begin{theorem}
Assuming Conjecture~\ref{conjecture:mountain_end_minimum} is true, then in the mountain case, we can fill at least $\frac{1}{2} + \frac{1}{2n}$ of the square.
\end{theorem}

\begin{proof}
We prove this by induction on $n$. For our base case $n = 1$, we can fill the entire square, which is equal to $\frac{1}{2} + \frac{1}{2n} = 1$.

Now, for larger $n$, the permutation either starts with $1$, or ends with $(m - 1, m - 2, \ldots, 1)$ for some $m \geq 2$.

\paragraph*{Case 1: permutation starts with 1 \\}
The rectangular region with opposite vertices $P_1$ and $(1, 1)$ is a smaller copy of the mountain case with $n - 1$ points. By the inductive hypothesis, we can fill at least $\frac{1}{2} + \frac{1}{2(n - 1)}$ of this region. Then, by Lemma~\ref{lemma:start1}, we can fill at least $\frac{1}{2} + \frac{1}{2n}$ of the whole square.

\paragraph*{Case 2: permutations ends with $(m - 1, m - 2, \ldots, 1)$ \\}
There will be $n - m$ points in the rectangular region $R$ with opposite corners $P_1$ and $\left(x_{n - m + 1}, 1\right)$, and note that the points in $R$ form the mountain case with $n - m$ points. By the inductive hypothesis, if $k$ is the proportion of $R$ that we can fill, then $k \geq \frac{1}{2} + \frac{1}{2(n - m)}$, or equivalently, $n \geq m + \frac{1}{2k - 1}$.

Then assuming Conjecture~\ref{conjecture:mountain_end_minimum}, we can fill at least $\frac{1}{2} + \frac{1}{2\left(m + \frac{1}{2k - 1}\right)}$ of the whole square. As $n \geq m + \frac{1}{2k - 1}$, we have that this expression is at least $\frac{1}{2} + \frac{1}{2n}$, completing the inductive step.
\end{proof}

\section{Four dots}
\label{sec:four_dots}

We calculated the minimal area for all 6 possible permutations when $n = 4$.
There are six possible permutations of the three non-origin points. Two are covered by the increasing and decreasing cases. $(1, 3, 2)$ is a cliff permutation, $(2, 1, 3)$ corresponds to a split-layer permutation, $(2, 3, 1)$ corresponds to a mountain, and $(3, 1, 2)$ is the inverse of $(2, 3, 1)$.

\subsection{Permutation (1, 2, 3)}

This is the all increasing case with $n = 4$, so by Theorem~\ref{theorem:increasing_case}, the minimum area is $\frac{5}{8}$.

\subsection{Permutation (1, 3, 2)}
This is the cliff case with $n = 4$ and $m = 1$, so by Corollary~\ref{corollary:cliff_case}, we can fill at least $\frac{49}{76}$ of the square.

\subsection{Permutation (2, 1, 3)}
\begin{theorem}
\label{theorem:213}
When $n = 4$ and the permutation is $(2, 1, 3)$, the minimum area is $\frac{837 - 11\sqrt{33}}{1152}$, and this bound is tight when $S = \{(x^2, x), (x, x^2), (y, y)\}$, where $x = \frac{9 + \sqrt{33}}{24}$ and $y = \frac{33 + \sqrt{33}}{48}$.
\end{theorem}

\begin{proof}
First, note that $P_3$ is a splitting point, and let $R$ be the rectangular region above and to the right of $\dom\{P_1, P_2\}$. By Theorem~\ref{theorem:independent_layers}, we can pack $R$ independently from the rest of the unit square. Along with the lower left corner of $R$, the points in $R$ form the increasing case with 2 points, so we can fill at least $\frac{3}{4}$ of $R$.

Also, by Lemma~\ref{lemma:maximal_at_origin}, we have $x_1 = x_2y_1 = y_2$, so $y_1 = \frac{x_1}{x_2}$ and $y_2 = x_1$.

We will now compute the area. If we assume that we can fill all of the staircase region defined by $P_1$ and $P_2$, including all of $R$, then the area would be $1 - 2x_1 + \frac{x_1^2}{x_2} + x_1x_2$, by a similar computation as in the decreasing case, in Section~\ref{sec:decreasing}. The area of $R$ is $(1 - x_2)(1 - y_1) = 1 - \frac{x_1}{x_2} - x_2 + x_1$. Subtracting a quarter of the area of $R$ from the previous expression, we see that the total area is
\[\frac{3}{4} - \frac{9}{4}x_1 + \frac{1}{4}x_2 + \frac{x_1^2}{x_2} + x_1x_2 + \frac{x_1}{4x_2}.\]

The partial derivative of the area with respect to $x_1$ is $-\frac{9}{4} + \frac{2x_1}{x_2} + x_2 + \frac{1}{4x_2}$, and the partial derivative with respect to $x_2$ is $\frac{1}{4} - \frac{x_1^2}{x_2^2} + x_1 - \frac{x_1}{4x_2^2}$.

They are both zero at two points: $(x_1, x_2) = \left(\frac{1}{96}(19 - 3\sqrt{33}, \frac{1}{24}(9 - \sqrt{33}\right)$, and $(x_1, x_2) = \left(\frac{1}{96}(19 + 3\sqrt{33}, \frac{1}{24}(9 + \sqrt{33}\right)$.

Of these local minima, the second one gives the lower value of about 0.67171, versus the value of 0.78142 produced by the first solution.

\begin{figure}
    \centering
    \includegraphics[scale = 0.7]{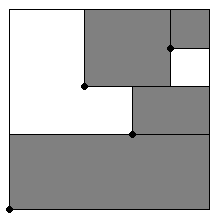}
    \caption{The set of points that minimizes the maximal area in the permutation $(2, 1, 3)$}
    \label{fig:213_optimal}
\end{figure}

Thus, the minimal area is $\frac{837 - 11\sqrt{33}}{1152}$, when
\[(x_1, x_2) = \left(\frac{1}{96}(19 + 3\sqrt{33}, \frac{1}{24}(9 + \sqrt{33}\right),\]
as shown in Figure~\ref{fig:213_optimal}.
\end{proof}

\begin{remark}
The points in this case all have irrational coordinates, unlike in previous cases. This shows that not all minimal configurations have rational points.
\end{remark}

\subsection{Permutation (2, 3, 1)}
We find the minimum area and equality case when the permutation is $(2, 3, 1)$.

\begin{theorem}
\label{theorem:start2_end1}
If the permutation starts with 2 and ends with 1, and we can fill a proportion $k$ of the rectangle with opposite vertices $P_1$ and $(x_{n - 1}, 1)$, then if $k \geq \frac{3}{4}$, we can fill at least
\[1 - \frac{4k}{3} + \frac{32k^2}{27} + \left(\frac{2}{9} - \frac{8k}{27}\right)\sqrt{16k^2 - 12k}\]
of the square, and if $k \leq \frac{3}{4}$, we can fill at least
\[1 - \frac{1}{4k}\]
of the square.
\end{theorem}

\begin{proof}
Note that $P_{n - 1}$ is in the final decreasing run, and the rectangles anchored at points in the final decreasing run form the staircase region. Since, by Lemma~\ref{lemma:final_staircase}, the staircase region is always filled when the area is maximized, we may assume that the rectangle anchored at $P_{n - 1}$ has $(1, 1)$ as its upper right corner, as shown in Figure~\ref{fig:231_diagram}.

\begin{figure}
    \centering
    \includegraphics[scale = 0.5]{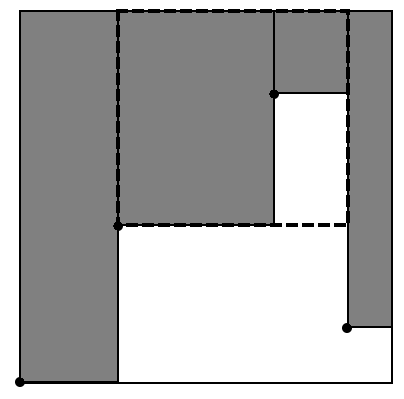}
    \caption{The highlighted region is equivalent to the 2 dots increasing case, so $k = \frac{3}{4}$ in this example.}
    \label{fig:231_diagram}
\end{figure}

Now, let $R$ be the rectangle with opposite vertices $P_1$ and $(x_{n - 1}, 1)$, shown as the dashed rectangle in Figure~\ref{fig:231_diagram}. Note that any packing on these points includes an induced packing of rectangle $R$, and we can always fill a proportion $k$ of $R$.

By Lemma~\ref{lemma:maximal_at_origin}, we have $x_1 = y_1x_{n - 1} = y_{n - 1}$, so $y_1 = \frac{x_1}{x_{n - 1}}$ and $y_{n - 1} = x_1.$

Thus, the area of the rectangle anchored at the origin is $x_1$, the area of the rectangle anchored at $P_{n - 1}$ has area $(1 - x_{n - 1})(1 - y_{n - 1}) = (1 - x_{n - 1})(1 - x_1)$, and the area of rectangle $R$ is $(x_{n - 1} - x_1)(1 - y_1) = (x_{n - 1} - x_1)\left(1 - \frac{x_1}{x_{n - 1}}\right).$

Therefore, the total area we can fill is
\[x_1 + (1 - x_{n - 1})(1 - x_1) + k(x_{n - 1} - x_1)\left(1 - \frac{x_1}{x_{n - 1}}\right).\]
This expression is equal to
\[1 - 2kx_1 - (1 - k)x_{n - 1} + k\frac{x_1^2}{x_{n - 1}} + x_1x_{n - 1}.\]

Taking the partial derivative with respect to $x_1$ and setting it to 0, we get
\[-2k + 2k\frac{x_1}{x_{n - 1}} + x_{n - 1} = 0.\]
Solving for $x_1$, we get
\[x_1 = x_{n - 1}\left(1 - \frac{x_{n - 1}}{2k}\right),\]
and we can check that the partial derivative goes from negative to positive at this value of $x_1$.

Taking the partial derivative of the area with respect to $x_{n - 1}$, we see that
\[-(1 - k) - k\frac{x_1^2}{x_{n - 1}^2} + x_1 = 0.\]
Substituting $x_1 = x_{n - 1}\left(1 - \frac{x_{n - 1}}{2k}\right)$, we see that
\[-1 + k - k\left(1 - \frac{x_{n - 1}}{2k}\right)^2 + x_{n - 1}\left(1 - \frac{x_{n - 1}}{2k}\right) = 0.\]
Expanding and collecting like terms, we get
\[\frac{3}{4k}x_{n - 1}^2 - 2x_{n - 1} + 1 = 0.\]

Solving for $x_{n - 1}$, and using the fact that $x_{n - 1} < 1$, we get
\[x_{n - 1} = \frac{4k - \sqrt{16k^2 - 12k}}{3}.\]
Note that the discriminant is negative if $k < \frac{3}{4}$. That means that the partial derivative with respect to $x_{n - 1}$ is always negative, meaning the minimum area is attained at $x_{n - 1} = 1$. Also, if $k = \frac{3}{4}$, then we get $x_{n - 1} = 1$.

Since $x_{n - 1} = 1$, any rectangle anchored at it has 0 area. Thus, it is equivalent to there being $n - 1$ points satisfying a permutation that starts with 1. Furthermore, we can fill a proportion $k$ of the area of the rectangle with opposite vertices $P_1$ and $(x_{n - 1}, 1) = (1, 1)$.

Therefore, we can apply Lemma~\ref{lemma:start1} to see that we can fill at least $1 - \frac{1}{4k}$ of the area of the whole square.

Now, assume $k \geq \frac{3}{4}$. If we plug in $x_1 = x_{n - 1}\left(1 - \frac{x_{n - 1}}{2k}\right)$ into our formula for the area, we get
\[1 - x_{n - 1} + x_{n - 1}^2 - \frac{x_{n - 1}^3}{4k}.\]

Plugging in $x_{n - 1} = \frac{4k - \sqrt{16k^2 - 12k}}{3}$, we get the minimum area is
\[1 - \frac{4k}{3} + \frac{32k^2}{27} + \left(\frac{2}{9} - \frac{8k}{27}\right)\sqrt{16k^2 - 12k}.\]
\end{proof}

A graph of the minimum area for various values of $k$ is shown in Figure~\ref{fig:start2end1_graph}. Note that it is a piecewise function, split between $k \leq \frac{3}{4}$ and $k \geq \frac{3}{4}$. This function is continuous, as both sides give $\frac{2}{3}$ for $k = \frac{3}{4}$. In fact, it is differentiable, as both sides have derivative $\frac{4}{9}$ at $k = \frac{3}{4}$. It is not twice-differentiable, as the right side has no second derivative.

\begin{figure}
\centering
\begin{tikzpicture}
\begin{axis}[xlabel = $k$, ylabel = Minimum area, ymin = 0.4, ymax = 0.8]
\addplot[domain = 0.5:0.75, color = black]{1 - 1/(4*x)};
\addplot[domain = 0.75:1, color = black]{1 - 4*x/3 + 32*x^2/27 + (2/9 - 8*x/27)*sqrt(16*x^2 - 12*x)};
\end{axis}
\end{tikzpicture}
\caption{The minimum area for various values of $k$}
\label{fig:start2end1_graph}
\end{figure}
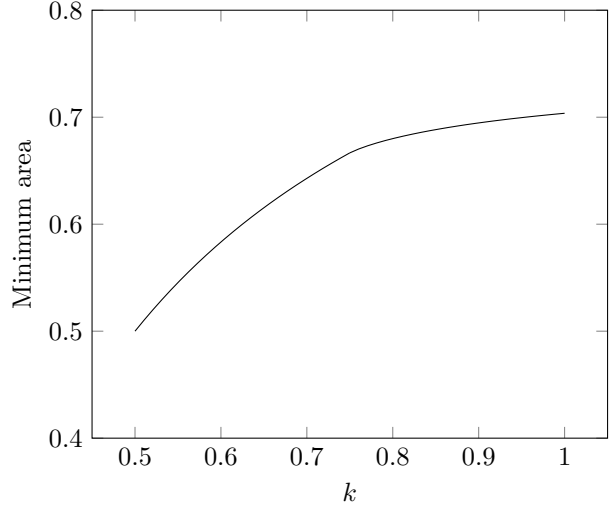

\begin{corollary}
When the permutation is $(2, 3, 1)$, we can always fill at least $\frac{2}{3}$ of the square, and this bound is tight if and only if $P_1 = \left(\frac{1}{3}, \frac{1}{3}\right)$, $P_2 = \left(\frac{2}{3}, \frac{2}{3}\right)$, and $x_3 = 1$.
\end{corollary}

\begin{proof}
If the permutation is $(2, 3, 1)$, then the rectangle with opposite vertices $P_1$ and $(x_3, 1)$ has an induced packing equivalent to the 2 increasing points case. Thus, by Theorem~\ref{theorem:increasing_case}, we can fill at least $\frac{3}{4}$ of this area. Applying Theorem~\ref{theorem:start2_end1} with $k = \frac{3}{4}$, we see that we can fill at least $\frac{2}{3}$ of the area.
\end{proof}

\begin{remark}
Unlike the case of increasing and decreasing runs, this case has many configurations with the same minimal area. Point $P_3$ can be anywhere on the line segment connecting $(1, 0)$ to $\left(1, \frac{1}{3}\right)$.
\end{remark}

\subsection{Permutation (3, 1, 2)}
Note that this permutation is the inverse of $(2, 3, 1)$, so the minimum area is the same.

\subsection{Permutation (3, 2, 1)}
This is the all decreasing case with $n = 4$, so the minimum area is $\frac{175}{256}$.

\subsection{The final bound for four dots}

This means that when $n = 4$, the worst case area is $\frac{5}{8}$, when \[S = \{(0, 0), \left(\frac{1}{4}, \frac{1}{4}\right), \left(\frac{1}{2}, \frac{1}{2}\right), \left(\frac{3}{4}, \frac{3}{4}\right)\}.\]

\section{Sparse decreasing case}
\label{sec:sparse_decreasing}

We now consider permutations that have special decreasing subsequences, but are not necessarily decreasing themselves.

\begin{definition}
For a permutation $(a_1, \ldots, a_{n - 1})$, we define its \emph{greedy decreasing subsequence} to be the sequence containing every element, including $a_1$, that is less than all preceding elements.
\end{definition}

Equivalently, we can build this subsequence starting from the end of the permutation. We include 1, and then pick the smallest number to the left of one. Continue recursively.

For example, the greedy decreasing subsequence of $(7, 5, 9, 6, 4, 1, 3, 2)$ is $(7, 5, 4, 1)$, and the greedy decreasing subsequence of any permutation starting with $1$ is just $(1)$.

\begin{definition}
For a set of points in the unit square, the \textit{untouchable points} are the points corresponding to the greedy decreasing subsequence of the permutation.
\end{definition}

Note that only rectangles anchored at the origin can be bounded by the untouchable points. No other maximal rectangles can touch the untouchable points.

Every other point, excluding the origin, is above and to the right of at least one untouchable point.

\begin{definition}
We say that a permutation $(a_1, a_2, \ldots, a_{n - 1})$ is \emph{$m$-sparse decreasing} if there are no subsequences of $m$ or more consecutive terms, none of which are in the greedy decreasing subsequence.
\end{definition}

Informally, a permutation is $m$-sparse decreasing if we can delete consecutive strings of fewer than $m$ elements to get the greedy decreasing subsequence.

Note that 1-sparse decreasing permutations are simply decreasing permutations.

For example, $(2, 1, 4, 3)$ is 3-sparse, as we can delete the last two elements to get $(2, 1)$, which is the greedy decreasing subsequence. However, $(2, 1, 4, 3)$ is not 2-sparse, as we would need to delete both $4$ and $3$, which are 2 consecutive elements.

We now show that we can always fill some fraction, dependent only on $m$, of the staircase region defined by the greedy decreasing subsequence, for any $m$-sparse permutation.

\begin{lemma}
\label{lemma:sparse_decreasing_staircase}
If the permutation is $m$-sparse decreasing, and we can fill at least $k_m$ of any square with $m$ or fewer points including the origin, then we can fill at least $k_m$ of the staircase region defined by the untouchable points.
\end{lemma}

\begin{proof}
Partition the staircase region into vertical rectangles anchored at the untouchable points, as shown in Figure~\ref{fig:almost_decreasing}.

\begin{figure}
    \centering
    \includegraphics{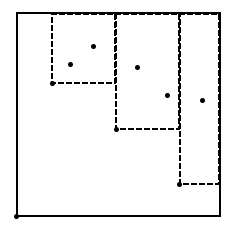}
    \caption{Example with $m = 3$}
    \label{fig:almost_decreasing}
\end{figure}

As the permutation is $m$-sparse decreasing, each rectangle will have fewer than $m$ points in its interior. Counting the point at the lower left corner of each rectangle, we see that there are at most $m$ points in each rectangle. Thus, we can fill at least $k_m$ of each rectangle, meaning we can fill at least $k_m$ of the entire staircase region.
\end{proof}

We now prove a similar result on how much of the whole square we can fill.

\begin{theorem}
\label{theorem:k_l_staircase}
If we can fill a proportion $k$ of the staircase region, and there are $\ell$ points in the greedy decreasing subsequence, then we can fill at least
\[k\left(1 - \left(1 - \frac{1}{k(\ell + 1)}\right)^{\ell + 1}\right)\]
of the whole square.
\end{theorem}

\begin{proof}
Let $\ell$ be the number of untouchable points, and let $(h_i, v_i)$ be the $i^{th}$ untouchable point.

Similarly as in the decreasing case in Section~\ref{sec:decreasing}, we have that the maximal rectangles anchored at the origin have upper right corners $(h_1, 1)$, $(h_2, v_1)$, $\ldots$, $(h_{\ell}, v_{\ell - 1})$, $(1, v_{\ell})$. We have from Lemma~\ref{lemma:maximal_at_origin} that all these rectangles have the same area.

By the same computation as in the decreasing case, we find that the area of the staircase region is
\[1 - \left(\ell + 1\right)h_1 + h_1\left(\frac{h_1}{h_2} + \frac{h_2}{h_3} + \cdots + \frac{h_{\ell - 1}}{h_{\ell}} + h_{\ell}\right).\]
As we can only fill $k$ of the staircase region, the total area, when we add the rectangle anchored at the origin, is
\[h_1 + k\left(1 - \left(\ell + 1\right)h_1 + h_1\left(\frac{h_1}{h_2} + \cdots + \frac{h_{\ell - 1}}{h_{\ell}} + h_{\ell}\right)\right).\]
If $h_1$ is fixed, we see that this expression is minimized when
\[\frac{h_1}{h_2} = \frac{h_2}{h_3} = \cdots = \frac{h_{\ell - 1}}{h_{\ell}} = h_{\ell}.\]
This means that $h_1 = h_{\ell}^{\ell}$, and the area is
\[k + (1 - k(\ell + 1))h_{\ell}^{\ell} + k\ell h_{\ell}^{\ell + 1}.\]

Taking the derivative with respect to $h_{\ell}$, we see that this expression is minimized when
\[\ell(1 - k(\ell + 1))h_{\ell}^{\ell - 1} + k\ell(\ell + 1)h_{\ell}^{\ell} = 0,\]
or
\[1 - k(\ell + 1) + k(\ell + 1)h_{\ell} = 0,\]
meaning
\[h_{\ell} = 1 - \frac{1}{k(\ell + 1)}.\]
Plugging in this value of $h_\ell$, we see that the area is at least
\[k + (1 - k(\ell + 1))\left(1 - \frac{1}{k(\ell + 1)}\right)^{\ell} + k\ell \left(1 - \frac{1}{k(\ell + 1)}\right)^{\ell + 1}.\]
This expression can be simplified to
\[k\left(1 - \left(1 - \frac{1}{k(\ell + 1)}\right)^{\ell + 1}\right).\]
\end{proof}

\begin{corollary}
\label{corollary:k_staircase}
If we can fill a proportion $k$ of the staircase region, then we can fill at least
\[k\left(1 - e^{-\frac{1}{k}}\right)\]
of the whole square.
\end{corollary}
\begin{proof}
By Theorem~\ref{theorem:k_l_staircase}, we can fill at least $k\left(1 - \left(1 - \frac{1}{k(\ell + 1)}\right)^{\ell + 1}\right)$ of the square.

Because this expression is decreasing as $\ell$ increases, and approaches
\[k\left(1 - e^{-\frac{1}{k}}\right)\]
as $\ell$ approaches infinity, we see that the area is at least
\[k\left(1 - e^{-\frac{1}{k}}\right).\]
\end{proof}

Using this result, we prove the Main Conjecture~\ref{conjecture:main} for 3-sparse decreasing permutations.

\begin{corollary}
If the permutation is 3-sparse decreasing, then we can fill at least $\frac{2}{3}\left(1 - e^{-\frac{3}{2}}\right) \approx 0.5179$ of the whole square.
\end{corollary}

\begin{proof}
We have shown at the end of Section~\ref{sec:decreasing} that we can fill at least $\frac{2}{3}$ of any square with three dots in it. Thus, by Lemma~\ref{lemma:sparse_decreasing_staircase}, we can fill at least $\frac{2}{3}$ of the staircase region. Plugging this into Corollary~\ref{corollary:k_staircase}, we see that we can fill at least $\frac{2}{3}\left(1 - e^{-\frac{3}{2}}\right)$ of the whole square.
\end{proof}

\subsection{9 or fewer points}
\label{subsection:9_dots}
Using our results in the sparse decreasing case, we prove the Main Conjecture~\ref{conjecture:main} in the case that $n \leq 9$.

\begin{theorem}
If $n \leq 9$, then we can fill at least $\frac{1}{2}$ of the whole square.
\end{theorem}

\begin{proof}
If the greedy decreasing subsequence has $\ell$ points in it, and the permutation is $m$-sparse decreasing, then there are $n - \ell - 1$ points that are not the origin or in the greedy decreasing subsequence. Thus, if the greedy decreasing subsequence has $\ell$ elements, then the permutation, in the worst case, is $(n - \ell)$-sparse decreasing.

Now, define $k_n$ to be the minimum amount of the square we can fill if there are $n$ points, and note that $k_1 = 1$.

By Lemma~\ref{lemma:sparse_decreasing_staircase}, we can fill at least $k_{n - \ell}$ of the staircase region. Plugging this into Theorem~\ref{theorem:k_l_staircase}, and letting $\ell$ range from $1$ to $n - 1$, we see that
\[k_n \geq \min_{1 \leq \ell \leq n - 1}\left(k_{n - \ell}\left(1 - \left(1 - \frac{1}{k_{n - \ell}(\ell + 1)}\right)^{\ell + 1}\right)\right).\]

This allows us to compute the following table of lower bounds for $k_n$.

\begin{center}
    \begin{tabular}{c|c}
        $n$ & $k_n$ \\ \hline
        1 & 1.0000 \\
        2 & 0.7500 \\
        3 & 0.6667 \\
        4 & 0.6250 \\
        5 & 0.5833 \\
        6 & 0.5615 \\
        7 & 0.5374 \\
        8 & 0.5211 \\
        9 & 0.5080 \\
        10 & 0.4934
    \end{tabular}
\end{center}

This means that for $n \leq 9$, we have $k_n \geq \frac{1}{2}$.
\end{proof}

Note that this shows that $k_n = \frac{1}{2} + \frac{1}{2n}$ for $n = 1, 2, 3, 4$.

\section{Acknowledgments}
I am grateful to Dr.~Tanya Khovanova for being my mentor, guiding me through this research project, providing many useful suggestions, and helping me write and edit this paper.

I am grateful to Prof.~Yufei Zhao for suggesting this project to me, and for very useful discussions. I also want to thank the PRIMES program for giving me this opportunity.

\printbibliography
\end{document}